\documentclass[12pt]{amsart}
\usepackage{amssymb}
\usepackage{enumerate}
\usepackage[all]{xy}
\usepackage{pb-diagram}

\newtheorem{conjecture}[equation]{Conjecture}
\newtheorem{thm}[equation]{Theorem}
\newtheorem{prop}[equation]{Proposition}
\newtheorem{cor}[equation]{Corollary}

\newtheorem{lemma}[equation]{Lemma}

\numberwithin{equation}{section}

\newcommand{\Q}{\mathbb Q}
\newcommand{\Z}{\mathbb Z}
\newcommand{\R}{\mathbb R}

\newcommand{\Sc}{\mathbb S}
\newcommand{\VV}{\mathbb V}
\newcommand{\WW}{\mathbb W}

\newcommand{\F}{\mathbb F}

\newcommand{\A}{\mathbb A}

\newcommand{\N}{\mathbb N}

\def\Hom{{\rm Hom }}
\def\Aut{{\rm Aut}}
\def\ord{{\rm ord}}

\def\SL{{\rm SL}}
\def\disc{{\rm disc}}

\def\Sp{{\rm Sp}}

\def\GU{{\rm GU}}

\newcommand{\U}{{\rm U}}

\def\GL{{\rm GL}}

\def\Sp{{\rm Sp}}

\def\A{{\mathbb A}}

\def\CC{{\mathbb C}}

\def\R{{\mathbb R}}

\def\Z{{\mathbb Z}}

\title[Restriction Problems for Classical Groups]{Restrictions of  representations of classical groups: examples}

\author{Wee Teck Gan, Benedict H. Gross and Dipendra Prasad}

\address{W.T.G.: Department of Mathematics, University of California at San Diego, 9500 Gilman Drive, La Jolla, 92093} \email{wgan@math.ucsd.edu}
\address{B.H.G: Department of Mathematics, Harvard University, 
Cambridge, MA 02138}\email{gross@math.harvard.edu}
\address{D.P.: School of Mathematics, Tata Institute of Fundamental
Research, Colaba, Mumbai-400005, INDIA}
\email{dprasad@math.tifr.res.in}

\begin{document}
\maketitle

\tableofcontents
\section{Introduction}

This paper is a sequel to [GGP], where we considered several restriction problems in the 
representation theory of classical groups over local and global fields.  
Assuming the Langlands-Vogan parameterization of irreducible representations, we formulated
 precise conjectures for the solutions of these restriction problems. In the local case, our conjectural answer is given in terms of Langlands parameters and certain natural symplectic root numbers associated to them. In the global case, the conjectural answer is expressed in terms of 
the central critical value or derivative of a global $L$-function. 
For the precise statements of the restriction problems and our conjectures, we refer the reader to [GGP].

\vskip 10pt

The conjectures for the case of special orthogonal groups were contained in the earlier papers [GP1] and [GP2] and were suggested by the results of Waldspurger [Wa1,2,3], Tunnell-Saito [Tu], [Sa],  and Prasad  [P1, 2, 3]
in certain low rank cases. Since then, there have been further results in the orthogonal case, both locally and globally; see, for example  [P4], [GR],  [GJS2], and [PT].  Most notably, in two recent preprints [Wa4] and [Wa5], Waldspurger has made spectacular progress towards to a full resolution of the local conjectures of [GP1, GP2].
\vskip 10pt

In this paper, we provide some evidence for the conjectures of [GGP] in the unitary case. More precisely, we shall consider the restriction problems in the following cases:

 \vskip 10pt

\begin{enumerate}[(i)]

\item the depth zero supercuspidal $L$-packets of DeBacker-Reeder[ DR], which are associated to tame regular discrete $L$-parameters;
\vskip 5pt

\item certain low rank cases, such as $\U(1) \times \U(1)$, $\U(1) \times \U(2)$, $\U(2) \times \U(2)$ and $\U(2) \times \U(3)$.
\end{enumerate}
\vskip 5pt

\noindent In each case, we shall establish [GGP, Conjecture 16.3].

 \vskip 15pt
 \noindent{\bf Acknowledgments:}  W. T. Gan is partially supported by NSF grant DMS-0801071. 
B. H. Gross is partially supported by NSF grant DMS 0901102. 
 D. Prasad was partially  supported by a  Clay Math Institute fellowship during the course of this work. We also thank P. Deligne, S. Kudla, G. Lusztig, M. Reeder, D. Rohrlich, and J.-L. Waldspurger for their help.
 \vskip 15pt
  \section{Discrete series parameters}\
 
We begin with  the computation of  the distinguished character in [GGP, Conjecture 16.3]
$$
\chi = \chi_N \times \chi_M : A_M \times A_N  \rightarrow  <\pm 1>,
$$
which is defined using local root numbers, for some discrete series parameters of the unitary groups
$G = \U(W) \times \U(W_0)$, associated to the quadratic extension of local fields $k/k_0$.
 
\vskip 10pt

In general,  these discrete series parameters have the form
\begin{eqnarray*}
M & = & \bigoplus_i M_i \\
N & = & \bigoplus_j N_j
\end{eqnarray*}
where the $M_i$ are distinct conjugate-symplectic representations and the $N_j$ are
distinct conjugate-orthogonal representations of the Weil-Deligne group of $k$. The
dimension of $M$ is even and the dimension of $N$ is odd. In this case, the centralizer
$C_M \times C_N$ of the Langlands parameter is finite.
\vskip 5pt

We will only consider the
case here where each $M_i = M(\alpha_i)$ and each $N_j = N(\beta_j)$ is one dimensional.
Then $\alpha_i$ is a character of $k^\times /\N k^\times$ with $\alpha_i|_{k_0^{\times}}=
\omega_{k/k_0}$, and $\beta_j$ is a character of $k^\times /k_0^\times$. In this case,
we have the component groups
\begin{eqnarray*}
A_M  & = & \bigoplus \Z/2\Z \cdot e_i \\
A_N  & = & \bigoplus \Z/2\Z \cdot f_j.
\end{eqnarray*}
These vector spaces have dimension equal to $\dim M$ and $\dim N$ over $\Z/2\Z$, which is as large as possible. We have
\begin{eqnarray*}
M^{e_i=-1} & = & M(\alpha_i) \\
N^{f_j=-1} & = & N(\beta_j).
\end{eqnarray*}

Fix a nontrivial additive character $\psi_0$ of $k$ which is trivial on $k_0$. By the definition
of the character $\chi$, we have the formulae
\begin{eqnarray*}
\chi(e_i) & = & \epsilon(M(\alpha_i) \otimes N, \psi_0) \\
\chi(f_j)  & = & \epsilon(M \otimes N(\beta_j), \psi_0).
\end{eqnarray*}
Using the additivity of the local epsilon factors, this becomes
\begin{eqnarray*}
\chi(e_i) & =& \prod_k \epsilon(\alpha_i\beta_k, \psi_0) \\
\chi(f_j) & =& \prod_k \epsilon(\alpha_k \beta_j, \psi_0).
\end{eqnarray*}
\vskip 10pt

Since the products $\alpha_i\beta_j$ are all conjugate-symplectic characters of $k^\times$, we
need a formula to compute their root numbers. We will do this is two different cases 
- when $k/k_0 = \CC/\R$, which we take up now,
and then when $k/k_0$ is unramified which we do in the next section.

\vskip 5pt

\begin{prop} \label{P:11.1}
Assume that $k_0 = \R$ and choose an isomorphism $z: k \rightarrow \CC$. Let
$$
 \alpha = z^{-2a} \cdot (z\bar{z})^a = (\bar{z}/z)^a
$$
be a conjugate-symplectic character of $k^{\times}$, where $a$  is a half integer, and let
$$
\psi_0 = e^{2 \pi i {\rm Tr}(iz)} = e^{2 \pi (\bar{z} - {z})}
$$
Then 
\[  \epsilon(\alpha, \psi_0) = \begin{cases}
\text{$+1$  if $a > 0$;} \\
 \text{$-1$ if $a < 0$.}
 \end{cases} \]
\end{prop}

\vskip 5pt 

\begin{proof}
Tate [T, (3.2.5)] gives the formula
$$ \epsilon(\alpha, \psi) = i^{2a}$$
when $a > 0$ and $\psi = e^{2 \pi i {\rm Tr}(z)}$. Since $\psi_0(z) = \psi(iz)$, we find
$$\epsilon(\alpha, \psi_0) = i^{2a}\cdot \alpha(i) = +1$$
in this case. When $a < 0$ we must conjugate the isomorphism $z:k \rightarrow \CC$ to use
Tate's formula. This changes the character $\psi_0$, and hence the sign of $\epsilon$.
\end{proof}
\vskip 5pt

\begin{cor}
Assume that $k_0 = \R$, choose an isomorphism $z: k \rightarrow \CC$, and let $\psi_0 = e^{2 \pi (z - \bar{z})}$. If $M$ is the sum of the distinct symplectic characters $\alpha_i = (\bar{z}/z)^{a_i}$, where each $a_i$ is a half integer, and $N$ is the sum of the distinct orthogonal characters $\beta_j = (\bar{z}/z)^{b_j}$, where each $b_j$ is an integer, then
\begin{eqnarray*}
\chi(e_i) & =& (-1)^{m_i} \\
\chi(f_j) & =& (-1)^{n_j} 
\end{eqnarray*}
where
\begin{eqnarray*}
 m_i &= & \# \{r: a_i+b_r < 0\}  \\
n_j &= &\# \{r: a_r + b_j < 0\}.
\end{eqnarray*}
\end{cor}

\vskip 10pt

Finally, we note that in the case when $k_0=\R$, we may order the distinct characters $\alpha_i$
and $\beta_j$ in the parameter $\varphi$ so that 
$$\begin{matrix} a_1 > a_2 > a_3 \cdots 
{\sf~~~~~~~~~~~~in~~~} \frac{1}{2}\Z -
\Z \\
b_1 > b_2 >b_3 \cdots {\sf~~~~~~~~~~~~in~~~} \Z.
\end{matrix}$$
\vskip 10pt

\begin{cor} For $i < j$, we have 
\begin{eqnarray*} 
\chi(e_i) \chi(e_j) & = & (-1)^{m_{ij}} \\
\chi(f_i) \chi(f_j) & = & (-1)^{n_{ij}}. 
\end{eqnarray*}
where
\begin{eqnarray*}
 m_{ij} &= & \# \{r: a_i+b_r > 0 > a_j+b_r\}  \\
n_{ij} &= &\# \{r: b_i+a_r > 0 > b_j+a_r\}.
\end{eqnarray*}
\end{cor}
\vskip 10pt

Since we know how to describe the representations in the $L$-packets of discrete series
parameters when $k_0 = \R$ [GR], the calculation of $\chi(e_i)\chi(e_j)$ and $\chi(f_i)\chi(f_j)$ allows us 
to say something about the representation $\pi = \pi(\varphi,\chi) = \pi_1 \otimes \pi_2 $ of $G(\R)$ with $d(\pi) = 1$. The irreducible representations $\pi_1$ and $\pi_2$ 
are discrete series representations of even and odd dimensional unitary groups, with
infinitesimal characters
\begin{eqnarray*}
&& a_1 > a_2 >  a_3 > \cdots \\
&& b_1 > b_2 > b_3 > \cdots 
\end{eqnarray*}
in $X^{\star} + \rho$ respectively. Moreover, in the chambers defined by their
Harish-Chandra parameters, the simple root walls corresponding to 

\begin{eqnarray*}
e_i-e_{i+1} {\rm~~is ~~compact~~} & \Longleftrightarrow & \chi(e_i)\cdot  \chi(e_{i+1})=-1 \\
f_i-f_{i+1} {\rm~~is ~~compact~~} & \Longleftrightarrow & \chi(f_i) \cdot \chi(f_{i+1})=-1.
\end{eqnarray*}
More generally, for $i < j$ the positive root  
\begin{eqnarray*}
e_i-e_{j} {\rm~~is ~~compact~~} & \Longleftrightarrow & \chi(e_i) \cdot \chi(e_{j})=(-1)^{i+j} \\
f_i-f_{j} {\rm~~is ~~compact~~} & \Longleftrightarrow & \chi(f_i) \cdot \chi(f_{j})=(-1)^{i+j}.
\end{eqnarray*}
This determines the signature of the unitary group $G(\R)$, and in almost all cases 
the discrete series representation $\pi$.

\vskip 10pt

We end this section with a remark about branching from $\U(n,1)$ to $\U(n)$.
According to a theorem of Harish-Chandra, an irreducible admissible
$(\mathfrak g,K)$-module is determined by the action of $U(\mathfrak g)^K$,
where $U(\mathfrak g)$ denotes the universal enveloping algebra of $G$, on a given $K$-type
which appears in the representation space. Further, the action of $K
\times U(\mathfrak g)$ on the corresponding isotypical component
is irreducible. By a theorem of Kostant, for 
\[  G=\U(n,1) \quad \text{and} \quad  K=\U(n)\times \U(1), 
\]
 $U(\mathfrak g)^K$ is generated by the centers of the universal enveloping
algebras of $G$ and $K$, and thus is, in particular, abelian. 
This proves that any irreducible representation of $\U(n)$ appears with multiplicity at most one in any
irreducible representation of $\U(n,1)$. 

 \vskip 15pt

\section{Depth zero supercuspidals}  \label{S:depth0}

In this section, we test the restriction conjecture for some 
tamely ramified discrete parameters $\varphi$ of unitary groups. We begin by calculating the 
local root numbers, assuming that $k_0$ is
non-archimedean and  $k$ is the unramified quadratic extension of $k_0$.
\vskip 5pt

\begin{prop} \label{P:2.4}
Assume that $k_0$ is non-archimedean, and let $k$ be the unramified quadratic field extension of $k_0$. Let $\psi_0$ be an additive character of $k$ which is trivial on both $k_0$ and the maximal ideal of the ring of integers $A_k$, but is nontrivial on $A_k$. Let $\alpha$ be a conjugate-symplectic character of $k^{\times}$ of conductor $f(\alpha)$. Then
$$\epsilon(\alpha, \psi_0) = (-1)^{f(\alpha) + 1}.$$
 \end{prop}
\vskip 5pt

\begin{proof}
When $k/k_0$ is unramified, every conjugate-symplectic character $\alpha$ has the form
$$\alpha = \beta \cdot \mu,$$
where $\beta: k^{\times}/k_0^{\times} \rightarrow \CC^{\times}$ is a conjugate-orthogonal 
character and $\mu$ is the unramified quadratic character (which is conjugate-symplectic).
By [GGP, Section 5] and [FQ, Theorem 3], we have
\[ \epsilon(\beta,\psi_0) = +1 \]
for any character $\psi_0$ of $k$ which is trivial on $k_0$. Since $\mu$ is unramified, we have [T,(3.4.6)]
$$\epsilon(\alpha,\psi_0) = \epsilon(\beta,\psi_0) \cdot \mu(\pi^{f(\beta) + n(\psi_0)}).$$
Since $f(\beta) = f(\alpha)$ and $n(\psi_0) = -1$, this gives the formula in the proposition.
\end{proof}
\vskip 10pt

\begin{cor} \label{C:2.5}
 Assume that $k_0$ is non-archimedean, and let $k$ be the unramified quadratic field extension of $k_0$. Let $\psi_0$ be an additive character of $k$ which is trivial on both $k_0$ and the maximal ideal of the ring of integers $A_k$, but is nontrivial on $A_k$. Let 
 \[  M =\oplus_i \alpha_i \quad \text{and} \quad N = \oplus_j \beta_j \]
 where the$\alpha_i$'s are mutually distinct,  tamely ramified, conjugate-symplectic characters, and the $\beta_j$'s are mutually distinct, tamely ramified, conjugate-orthogonal characters. Order these characters so that
$$\alpha_1\beta_1= \alpha_2 \beta_2= \cdots =  \alpha_p \beta_p = \mu,$$
for $p \geq 0$ and no other products $\alpha_i \beta_j = \mu$. Then
\[ \chi(e_i) = \begin{cases}
\text{$-1$ when $i \leq p$;} \\
\text{$+1$ when $i > p$.}
\end{cases} \] 
Similarly,
\[ \chi(f_j) = \begin{cases}
\text{$-1$ when $j \leq p$;} \\
\text{$+1$ when $j > p$.} 
\end{cases} \]
Finally, $\chi(-1,1) = \chi(1,-1) = (-1)^p.$
\end{cor}

\vskip 5pt

\begin{proof}
Since our characters are all tamely ramified, we find
$$f(\alpha_i\beta_j) = 1,$$
unless $i = j \leq p$, in which case the product is equal to the unramified
character $\mu$ and $f(\alpha_i\beta_i) = 0$. Taking the product of epsilon
factors giving $\chi$ gives the desired result.
\end{proof}

\vspace{5mm}

We take the parameter 
\begin{eqnarray*}
M & = & \bigoplus M(\alpha_i) \\
N & = & \bigoplus N(\beta_j)
\end{eqnarray*}
given by the sum of distinct conjugate-symplectic and distinct conjugate-orthogonal
characters of $k^{\times}$. We will assume that all of these characters
are tamely ramified:
$$f(\alpha_i) = f(\beta_j) = 1.$$
\vskip 10pt

The $L$-packet $\Pi_{\varphi}$ of depth zero supercuspidal representations of the
pure inner forms $G = \U(W) \times \U(W_0)$ has been constructed by DeBacker and
Reeder [DR], and we briefly summarize their results in this case. Let $V$ be a hermitian space
of dimension $n$ over $k$. A parameter $\varphi$ of the above type for the unitary group
$\U(V) = \U_n$ gives, by restriction to the units of $k^{\times}$, a regular complex character $\rho$ of
the anisotropic torus $S = (\U_1)^n$. Each embedding
$$ f : S \rightarrow \U(V)$$
corresponds to a decomposition of the space $V$ into the direct sum of orthogonal lines, stable
under the action of $S$. Write this decomposition, choosing a basis for each line, as
$$ V = \bigoplus k v_i.$$
The $\U(V)$-conjugacy class of the embedding $f$ depends only on the signs
$$\epsilon_i = (-1)^{\ord \langle v_i,v_i \rangle},$$
which must satisfy the one relation
$$\prod_i \epsilon_i = (-1)^{\ord(\disc V)}.$$
Since the two hermitian spaces $V$ and $V'$ of dimension $n$ have distinct
hermitian discriminants, there are exactly $2^n$ conjugacy classes of embeddings
$f$ of $S$ into $\U(V)$ and $\U(V')$. These conjugacy classes correspond bijectively 
to the characters $\chi = \chi_f$ of the group $A_\varphi$, where $\chi(e_i) = \epsilon_i$.
\vskip 10pt

For each embedding $f: S \rightarrow \U(V)$, there is a 
unique maximal compact subgroup $K_f \subset \U(V)$ which contains
the image. This is the subgroup stabilizing the lattice,
$$L_f = \bigoplus A_k v_i,$$
where we normalize the basis vectors of our $S$-stable lines
to satisfy
$ 0 \leq {\rm ord} \langle v_i, v_i \rangle \leq 1$.
The reduction mod $\pi$
of $K_f$ has reductive quotient
$$\bar{K}_f \cong \U_p \times \U_{n-p}, $$
where $p$ is the number of $v_i$ with $(-1)^{{\rm ord} 
\langle v_i, v_i \rangle} = -1.$ Hence $K_f$ will be hyperspecial 
if and only if all of the inner products $ \langle v_i, v_i \rangle$ 
have valuations of the same parity. 
\vskip 10pt

The torus $S(q)$ embeds in $\bar{K}_f(q)$, and the regular tame character $\rho$
of $S(q)$ allows us to construct 
an irreducible, supercuspidal representation $R_f(S,\rho)$
of the finite group $\bar{K}_f(q)$, using the method of
Deligne and Lusztig. We view this as a representation of the
compact group $K_f$, and define the representation 
\[  \pi_\chi = \pi_f, \quad \text{of $\U(V)$} \]
as the compact induction of $R_f(S,\rho)$. These are the $2^n$
depth zero supercuspidal representation in the $L$-packet $\Pi_{\varphi}$.

\vskip 10pt

The Vogan bijection
between the set $\Pi_{\varphi}$ and the group of homomorphisms from $A_{\varphi}$
to $\langle\pm1\rangle$ is normalized as follows.
Assume that the hermitian space $V$ is split. The character $\chi = 1$ of
$A_{\varphi}$ gives rise to the non-degenerate hermitian $A$-lattice $L$ in $V$, with an orthogonal
basis whose inner products are units in $A$. Let $N_L$ be the unipotent radical of an Iwahori
subgroup of the hyperspecial maximal subgroup $K = \Aut(L)$ in $\U(V)$. The construction
of [GGP, \S 12] over the ring $A$ gives a surjective homomorphism
$$ f + f_0: N_L \rightarrow A^{n-1} + A(-)$$
where $A(-)$ is the eigenspace where $\sigma = -1$ on $A$, which consists of the
elements of trace $0$ to $A_0$.
\vskip 10pt

By [DR2], the character $\chi = 1$ of $A_{\varphi}$ corresponds to the unique representation
$\pi_1$ of in the $L$-packet of $\varphi$ which is induced from a generic, cuspidal
representation of the reductive quotient $\U_n(q)$ of $\U(L)$. All of the generic characters of
the unipotent radical $N(q)$ of a Borel subgroup of $\U_n(q)$ are conjugate, and we construct
one of them in the following manner.
\vskip 10pt

Let $\psi_0$ be an additive character of $k$ which is trivial on $k_0$ and the maximal
ideal $P$ of $A$, but is nontrivial on $A$. Since $A$ is unramified over $A_0$, we have
$$ A(+) + A(-) = A_0 + 2\cdot A.$$
Hence, for elements $z$ in $A(-)$, the character 
\[  z \mapsto \psi_0(z/2) \]
is nontrivial on $A(-)/P(-)$.
Let $\psi$ be any additive character that is trivial on $P$ but nontrivial on $A$. Then the
composition
$$\psi(\Sigma f(n)) \cdot \psi_0(f_0(n)/2): N_L \rightarrow \Sc^1$$
gives a generic character of $N_L$ whose reduction $\bmod \,P$ is a generic character of $N(q)$.
Hence the representation $\pi_1$ corresponding to the trivial character of $A_{\varphi}$ is
generic for the character obtained by scaling the additive character $\psi_0$  used to define
the root number character in Proposition \ref{P:2.4} and Corollary \ref{C:2.5} by the factor $1/2$, or equivalently by the factor $2$. This is the normalization predicted in [GGP,Conjecture 16.3].

\vskip 10pt

Now consider the parameter of $G = \U(W) \times \U(W_0) = \U_n \times \U_m$ which is given by
\begin{eqnarray*}
M & = & \bigoplus M(\alpha_i) \\
N & = & \bigoplus N(\beta_j)
\end{eqnarray*}
From the calculation of the character $\chi = \chi_N \times \chi_M$ of $A_\varphi$ in 
the previous section, we conclude that the irreducible representation $\pi_\chi$ of
$G = \U(W) \times \U(W_0)$ is compactly induced from a maximal compact subgroup
with reduction isomorphic to 
$$(\U_p \times \U_{n - p}) \times (\U_p \times \U_{m-p})$$
over $\F_q$. Here $p \geq 0$ is the number of pairs $(\alpha_i,\beta_i)$
with $\alpha_i\beta_i = \mu$. The finite dimensional representation that we
are inducing has the form
$$(R \otimes R(\alpha)) \otimes (R^{\vee} \otimes R(\beta))$$
where $R$ is the Deligne-Lusztig representation of $\U_p(q)$ associated
to the character $(\alpha_1,\alpha_2,...,\alpha_p)$ and $R^{\vee}$ is its dual
representation, associated to the character $(\beta_1,\beta_2,...,\beta_p)$.
The remaining representations $R(\alpha)$ of $\U_{n-p}(q)$ and $R(\beta)$
of $\U_{m-p}(q)$ are associated to characters whose components $\alpha_i$
and $\beta_j$ satisfy $\alpha_i\beta_j \neq \mu$ for all $i,j$.
\vskip 10pt

As support for [GGP, Conjecture 16.3], we will prove:
\vskip 5pt

\begin{thm} \label{T:depth-zero-sc}
Let $\pi_{\chi}$ be the depth zero supercuspidal representation of $G = \U(W) \times \U(W_0)$ defined above, which corresponds to the distinguished character in [GGP, conjecture 16.3]. Then
$\pi_{\chi}$ possesses a Bessel model, in the sense that
\[  \dim \Hom_H (\pi_{\chi},\nu) =1 \]
where $(H,\nu)$ is as defined in [GGP, \S 12].
\end{thm} 
\vskip 10pt

To prove the existence of a (unique) Bessel model for $\pi_\chi$, it is sufficient to
establish the existence of a Bessel model for the representation
$$R(\alpha) \otimes R(\beta) \quad \text{of $\U_{n-p} \times \U_{m-p}$}, $$ 
as there is clearly a unique $\U_p \times \U_p$
invariant linear form on $(R \otimes R^{\vee})$. We will do this in the following two sections,
after first studying the situation for general linear groups.

\vskip 15pt

\section{Branching laws for $\GL_n({\Bbb F}_q)$}
In this section, we calculate the restriction of a representation
of $\GL_n({\Bbb F}_q)$ to $\GL_{n-1}({\Bbb F}_q)$ where 
$\GL_{n-1}({\Bbb F}_q)$ sits inside $\GL_n({\Bbb F}_q)$ in the natural way as
$$A \mapsto \left(\begin{array}{cc} A & 0 \\  0 & 1 \end{array} 
\right ).$$
These branching laws are surely known in the literature, such as
in the work of Thoma [Th]; however, we have preferred to give a different
independent treatment.

\vskip 10pt

We begin by recalling the notion of twisted Jacquet functor. Let $P= M\cdot N$ be any
group such that $N$ is  a normal subgroup of $P$ and let $\varphi$ 
be a character of $N$ whose stabilizer in $M$ is denoted by $M_{\varphi}$. 
The data $(N, \varphi)$ defines the twisted Jacquet functor from the category of smooth representations of $P$ to the category of smooth representations of $M_{\varphi}$.
It associates to a 
representation $V$ of $P$  the largest quotient $V_{N, \varphi}$ of $V$ on which $N$ operates via
the character $\varphi$; clearly $V_{N,\varphi}$ is a representation space for 
$ M_\varphi$. The twisted Jacquet functor is exact. 
 \vskip 10pt
 
 Now let $E_{n-1}$ be the mirabolic subgroup of $\GL_n({\Bbb F}_q)$ 
consisting of matrices  whose last row is equal to $(0, 0,\cdots, 0, 1)$ and
let $N_n$ be the group of upper triangular unipotent matrices in $\GL_n( {\Bbb F}_q)$.
We fix a nontrivial character $\psi_0$ of $\mathbb{F}_q$ and let 
$\psi_n$ be the character of $N_n$, given by
\[ \psi_n(u)=\psi_0(u_{1,2}+u_{2,3}+\cdots + u_{n-1,n}).\] 
For a representation $\pi$ of $\GL_n({\Bbb F}_q)$, let  
\[  \pi^i = \text{ the {\it i}-th derivative of $\pi$}, \]
which is a representation of $\GL_{n- i}( {\Bbb F}_q)$. To recall the definition of $\pi^i$,  if 
\[  R_{n-i} = \GL_{n- i}({\Bbb F}_q) \cdot V_i \]
 is the subgroup of
$\GL_n({\Bbb F}_q)$ consisting of matrices 
$$\left(\begin{array}{cc} g & v \\  0 & z \end{array} \right)$$ 
with $g\in \GL_{n-i}( {\Bbb F}_q)$, $v\in M(n-i,i),z\in N_i$, and if the
character $\psi_i$ of $N_i$ is extended to $V_i$ by extending it trivially across 
$M(n-i, i)$,
then we have 
\[  \pi^i= \pi_{V_i, \psi_i} .\]
\vskip 10pt

If $\pi$ is
an irreducible  cuspidal representation of $\GL_n({\Bbb F}_q)$, 
then $\pi^i = \pi$ for $i=0$, and $\pi^n= 1$, the trivial representation
of the trivial group $\GL_0({\Bbb F}_q)$. All the other
derivatives of $\pi$ are 0.
\vskip 10pt

The following proposition is from Bernstein-Zelevinsky [BZ], where it was established for
non-archimedean local fields, but their proof works for finite fields as well.
It is known as the Leibnitz rule for derivatives.
\vskip 5pt

\begin{prop}  \label{P:BZ}
For $\pi_1$ a representation of $\GL_{n_1}({\Bbb F}_q)$  and $\pi_2$ of 
$\GL_{n_2}({\Bbb F_q})$, we let $\pi_1 \times \pi_2$
denote the representation of $GL_{n_1 + n_2}({\Bbb F}_q)$
 induced from the corresponding representation
of the parabolic subgroup with Levi subgroup 
$GL_{n_1}({\Bbb F}_q) \times  GL_{n_2}({\Bbb F}_q)$. 
Then there is a composition series of the $k$-th derivative 
$(\pi_1 \times \pi_2)^k$  whose
successive quotients are $ \pi_1^i \times \pi_2^{k-i}$ for $i=0,\cdots, k$.
\end{prop}

\vskip 5pt

Here is a generality from Bernstein and Zelevinsky [BZ].

\begin{prop}  \label{P:leibnitz}
 Any  representation $\Sigma$ of $E_{n-1}$ has a natural filtration of $E=E_{n-1}$
modules 
\[  0 \subset \Sigma_0 \subset  \Sigma_1 \subset  \Sigma_2 \subset \cdots \subset
\Sigma_n \]
such that
\[  \Sigma_{i+1}/\Sigma_i = {\rm ind}_{R_{i}}^E(\Sigma^{n-i} \otimes \psi_{n-i})  \quad \text{for $i=0,\cdots, n-1$,} \] 
where $R_{i} = \GL_i({\Bbb F}_q)\cdot V_{n-i}$ is the subgroup of $\GL_n({\Bbb F}_q)$ consisting 
 of 
$$\left(\begin{array}{cc} g & v \\  0 & z \end{array} \right)$$ 
with $g\in \GL_i({\Bbb F}_q)$, $v\in M(i,n-i),z\in N_{n-i}$, and the
character $\psi_{n-i}$ on $N_{n-i}$ is extended to $V_{n-i}$ by extending 
it trivially across 
$M(i,n- i)$.
\end{prop}

\vspace{2mm}

As a consequence of the above two propositions, we have the following corollary.
\vskip 5pt

\begin{cor}
Let $n = n_1 + \cdots + n_r$ be a sum of positive integers, and 
let $\pi_i$ be an irreducible cuspidal representation of $\GL_{n_i}({\Bbb F}_q)$ for $i = 1, \cdots , r$.
Let 
\[  \Pi= \pi_1 \times \cdots \times \pi_r \] 
be the corresponding 
parabolically induced
representation of $\GL_n({\Bbb F}_q)$. Then the restriction of $
\pi_1 \times \cdots \times \pi_r$ to $\GL_{n-1}({\Bbb F}_q)$ 
is a sum of the following representations:

$$\pi_{i_1} \times \pi_{i_2}\times \cdots \times \pi_{i_s} \times 
\Sigma[n-1-(i_1+\cdots + i_s)]$$
where $1 \leq i_1< i_2 < \cdots < i_s \leq r$ (the empty sequence is allowed) and 
$i_1+\cdots + i_s < n$. Moreover,
\[   \Sigma[m] = {\rm ind}_{N_m} ^{\GL_m(\mathbb{F}_q)} \psi_m  \]
denotes  the Gelfand-Graev representation of $\GL_m({\Bbb F}_q)$, with 
$\Sigma[1]$  equal to the regular representation of 
${\Bbb F}_q^\times$ and  $\Sigma[0]$ denoting the trivial
representation of the trivial group.
\end{cor}
\vspace{2mm}

\begin{proof} By Proposition  \ref{P:leibnitz}, the restriction of $\Pi$ to $E_{n-1}$ is the sum of 
\[  \Pi_{i+1}/\Pi_i = {\rm ind}^{E_{n-1}}_{R_i} (\Pi^{n-i} \otimes \psi_{n-i}). \]
Since $\GL_{n-1}({\Bbb F}_q) \cdot R_i =  E_{n-1}$ 
for any $i$, it follows that
\[  (\Pi_{i+1}/\Pi_i)|_{\GL_{n-1}(\mathbb{F}_q)} 
 =  \Pi^{n-i}  \times \Sigma[n-1-i] , \]
 where $\Sigma[n-1-i]$ is the Gelfand-Graev module of $\GL_{n-1-i}(\mathbb{F}_q)$. 
It only remains to calculate
the derivatives $\Pi^{n-i}$ of $\Pi$, but this follows from
Proposition \ref{P:BZ}.   
\end{proof}

\vspace{2mm}

As a simple consequence of this corollary, we have the following.

 \vskip 5pt
 
\begin{thm} \label{T:GL-finite}
Let $n = n_1 + \cdots + n_r$ be a sum of positive integers, and 
let $\pi_i$ be an irreducible cuspidal representation of $\GL_{n_i}({\Bbb F}_q)$, for $i = 1, \cdots , r$.
Let $n-1 = m_1 + \cdots + m_s$ be a sum of positive integers, and 
let $\mu_i$ be an irreducible cuspidal representation of $\GL_{m_i}({\Bbb F}_q)$.
Assume that the representations 
$\mu_1,\cdots, \mu_s$ are pairwise distinct, so that 
  the corresponding parabolically induced representation  $\mu_1 \times \cdots \times \mu_s$
of $\GL_{n-1}({\Bbb F}_q)$ is irreducible. 
Then 
\[  \dim  \Hom_{\GL_{n-1}} (\pi_1 \times \cdots \times \pi_r,   \mu_1 \times \cdots \times \mu_s) \] 
is equal to
$$ \prod_{i=1}^s  (1+m_i) \geq 1,$$
where  $m_i$ is the multiplicity with which $\mu_i$ appears in the set $\{ \pi_1 \cdots ,\pi_r\}$. 
In particular, if the $\pi_i$'s are mutually distinct as well, then 
\[  \dim  \Hom_{\GL_{n-1}} (\pi_1 \times \cdots \times \pi_r,   \mu_1 \times \cdots \times \mu_s) 
= 2^d \] 
where $d$ is the cardinality of the set 
\[   \{\pi_1, \cdots \pi_r \} \cap  \{ \mu_1,...,\mu_s \}. \]
\end{thm}
\vskip 15pt

\begin{cor}  \label{C:GL-finite2}
The restriction of the representation $\pi_1 \times \cdots \times \pi_r$ of $\GL_n(\mathbb{F}_q)$ contains the representation $\mu_1 \times \cdots \times \mu_s$ of $\GL_{n-1}(\mathbb{F}_q)$ 
(with $\mu_i$'s mutually distinct) with multiplicity one if and only if the sets  $\{ \pi_1, \cdots , \pi_r \}$ and $\{ \mu_1, \cdots , \mu_s \}$ have no common elements.

 \end{cor}

\vskip 15pt

\section{Branching laws for $\U_n({\Bbb F}_q)$}
In this section, we use the method of base change, also called Shintani descent, 
to deduce some conclusions about 
branching laws for the restriction of a representation of
$\U_n({\Bbb F}_q)$ to  $\U(n-1,{\Bbb F}_q)$ from the corresponding
results for general linear groups obtained in the previous
section. The result is then applied to give a proof of Theorem \ref{T:depth-zero-sc}.
\vskip 10pt

We make crucial use of the multiplicity 1 theorem for
restriction of representations of unitary groups over $p$-adic fields,  which was recently proved  by Aizenbud, Gourevitch, Rallis and Schiffmann in [AGRS].   A simple consequence of their result is:
 
\vskip 5pt

\begin{prop} \label{P:AGRS-finite}
Let $\pi_1$ be an irreducible  cuspidal representation of $\U_n(\mathbb{F}_q)$ and let 
\[  \pi_2 = I_P(\sigma) \]
be an irreducible  generalized principal series representation of $\U_{n+1}(\mathbb{F}_q)$, where $P$ is a parabolic subgroup of $\U_{n+1}(\mathbb{F}_q)$  and $\sigma$ is a cuspidal representation of 
a Levi factor of $P$. We allow the possibility that $P = \U_{n+1}$, in which case $\pi_2  =\sigma$ is cuspidal.  Then
\[  \dim \Hom_{\U_n(\mathbb{F}_q)}(\pi_2, \pi_1) \leq 1. \]
\end{prop}

\vskip 5pt

\begin{proof}
Let $k_0$ be  a local field with $\mathbb{F}_q$ as its residue field and let $k$ be 
its unramified quadratic extension. Then 
one can find quasi-split unitary groups $\U(W_0)$ and $\U(W)$ with $W_0 \subset W$, 
such that $\U(W_0) \times \U(W)$ over $k_0$ contains a hyperspecial maximal compact subgroup $K_0 \times K$ with reductive quotient 
 $\U_n(\mathbb{F}_q) \times \U_{n+1}(\mathbb{F}_q)$. Moreover, one may find a maximal parabolic subgroup $\tilde{P}$ of $\U(W)$, such that $\tilde{P} \cap K$ maps to the parabolic $P$ in the reductive quotient $\U_{n+1}(\mathbb{F}_q)$.  
 
 \vskip 5pt
 
 Let $\tilde{\pi}_1$ be a depth zero supercuspidal representations of  $\U(W_0)$
  which is obtained from $\pi_1$ by compact induction, so that
  \[  \tilde{\pi}_1 = {\rm ind}_{K_0}^{\U(W_0)} \pi_1 = {\rm Ind}_{K_0}^{\U(W_0)} \pi_1. \] 
 Similarly, let $\tilde{\sigma}$ be a depth zero supercuspidal representation of the Levi factor of $\tilde{P}$ which contains $\tau$ as a type. Then we may consider the generalized principal series representation $I_{\tilde{P}}(\tilde{\sigma})$ of $\U(W)$ which is irreducible for a generic choice of 
 $\tilde{\sigma}$. Moreover, if $K_1$ denotes the kernel of the natural projection map 
 \[  K \twoheadrightarrow \U_{n+1}(\mathbb{F}_q), \]
 then one has
 \[   I_{\tilde{P}}(\tilde{\sigma})^{K_1}
= I_P(\sigma). \]
 
 \vskip 5pt
 
 Now by Frobenius reciprocity, we have
 \begin{align}
 \dim \Hom_{\U(W_0)}( I_{\tilde{P}}(\tilde{\sigma}),  \tilde{\pi}_1) 
&=\dim \Hom_{K_0} ( I_{\tilde{P}}(\tilde{\sigma}), \pi_1) \notag \\
&= \dim \Hom_{\U_n(\mathbb{F}_q)}( I_{\tilde{P}}(\tilde{\sigma})^{K_{0,1}}, \pi_1)\notag
\end{align} 
where $K_{0,1}$ is the kernel of the projection map $K_0 \twoheadrightarrow \U_n(\mathbb{F}_q)$. Since $K_{0,1} \subset K_1$, we have
\[   I_{\tilde{P}}(\tilde{\sigma})^{K_{0,1}}  \supset   I_{\tilde{P}}(\tilde{\sigma})^{K_1}
= I_P(\sigma). \]
Thus we conclude that
\[  \dim \Hom_{\U(W_0)}( I_{\tilde{P}}(\tilde{\sigma}),  \tilde{\pi}_1) \geq
 \Hom_{\U_n(\mathbb{F}_q)}(\pi_2, \pi_1). \]
 By [AGRS],  the LHS is bounded above by $1$ for a generic choice of $\tilde{\sigma}$ (so that $I_{\tilde{P}}(\tilde{\sigma})$ is irreducible), and hence so is the RHS. This proves the proposition.
 \end{proof}
 
 \vskip 10pt
 
 A corollary of the above proposition is the uniqueness of Bessel models for cuspidal representations of unitary groups over finite fields. 
 \vskip 10pt

 \begin{prop} \label{P:AGRS-finite2}
 Let $\pi_1$ be an 
irreducible cuspidal representation of $\U_n({\Bbb F}_q)$, 
and let $\pi_2$ be an irreducible cuspidal representation
of $\U_m({\Bbb F}_q)$ with $n > m$ but $m \not \equiv  n \mod 2$.
\vskip 10pt

\noindent (i) Let $P$ be a maximal parabolic subgroup of $\U_{n+1}(\mathbb{F}_q)$ with Levi factor $\GL_r(\mathbb{F}_{q^2}) \times \U_m(\mathbb{F}_q)$ (so that $m+2r = n+1$) and let $\tau$ be a cuspidal representation of $\GL_r(\mathbb{F}_{q^2})$.
Consider the generalized principal series representation $I_P(\tau \boxtimes \pi_2)$. Then,  with the data $(H, \nu)$ defined as in [GGP, \S 12], we have

\[  \Hom_{H}(\pi_1 \boxtimes  \pi_2, \nu) \cong  \Hom_{\U_n(\mathbb{F}_q)}(I_P(\tau, \pi_2), \pi_1^{\vee}) \]
\vskip 10pt

\noindent (ii) We have:
\[  \dim \Hom_{H}(\pi_1 \boxtimes  \pi_2, \nu) \leq 1. \]
 \end{prop}
\vskip 5pt

\begin{proof}
\noindent (i) This is the finite field analog of [GGP, Theorem 15.1], with the same proof.
\vskip 5pt

\noindent (ii) If $n = m+1$, (ii) is a special case of Proposition \ref{P:AGRS-finite}.
In the general case when $n > m+1$, we choose
 $\tau$ in the context of (i) so that the induced representation $I_P(\tau \boxtimes \pi_2)$ is irreducible. Then (ii) follows immediately from (i) and Proposition \ref{P:AGRS-finite}.
\end{proof}

 \vskip 10pt

After the above propositions,  we shall study the restriction problem from $\U_n(\mathbb{F}_q)$ to $\U_{n-1}(\mathbb{F}_q)$ using Shintani descent. We begin by giving a brief review of this notion.
 
 \vskip 10pt
 Let $G$ be a connected 
reductive algebraic group over ${\Bbb F}_q$ and let $m \geq 1$
be a fixed integer. The group $G({\Bbb F}_{q^m})$ comes equipped
with its Frobenius automorphism $F$, whose set of fixed points is
$G({\Bbb F}_q)$. There is a natural map, called the norm mapping,
 \[  \{ \text{$F$-conjugacy 
classes in $G({\Bbb F}_{q^m})$} \} \longrightarrow   \{\text{conjugacy classes in $G({\Bbb F}_q)$}\} \]
which is a bijection. The norm mapping thus induces an isomorphism of vector spaces
\[  \{\text{class functions on $G(\mathbb{F}_q)$} \} \longrightarrow \{ \text{$F$-class functions on $G(\mathbb{F}_{q^m})$}\}, \]
which is called the base change map, and whose inverse is called Shintani descent.
Furthermore, the base change map is an isometry:

$$\langle \chi_1,\chi_2\rangle _{G({\Bbb F}_q)} =
 \langle \chi_1',\chi_2'\rangle _{G({\Bbb F}_{q^m})},$$
where $\chi_1$ and $\chi_2$ are class functions on $G({\Bbb F}_q)$
which are Shintani descents of
the $F$-class functions $\chi'_1$ and $\chi'_2$ 
on $G({\Bbb F}_{q^m})$. Here, the inner products are normalized so that
$$\langle 1,1\rangle _{G({\Bbb F}_q)} = \langle 1,1\rangle _{G({\Bbb F}_{q^m})}=1.$$

\vskip 10pt

According to Deligne-Lusztig, given a maximal torus $T$ 
of $G$ defined over ${\Bbb F}_q$, and a character 
\[ \theta: T({\Bbb F}_q) \rightarrow {\Bbb C}^{\times},\] 
there is a (virtual) representation of 
$G({\Bbb F}_q)$ denoted by $R(T,\theta)$, which is called a Deligne-Lusztig
representation. 
Now given a character $\theta$ as above, one has
the character 
\[ \theta_m: T_m=T({\Bbb F}_{q^m}) \rightarrow {\Bbb C}^{\times} \]
 obtained
by composing $\theta$ with the norm mapping: $T({\Bbb F}_{q^m}) \rightarrow T({\Bbb F}_q)$.
Thus one may consider the Deligne-Lusztig representation $R(T_m, \theta_m)$.
The following lemma is [DL, 5.16]:
\vskip 5pt

\begin{lemma} 
If $G$ has connected center, then if $R(T,\theta)$ is irreducible, so is 
$R(T_m, \theta_m)$.
\end{lemma}
\vskip 10pt

Henceforth, we assume that $G$ has connected center and $R(T, \theta)$ is irreducible.
The irreducible representation $R(T_m, \theta_m)$ is invariant by $F$ and thus can be extended (in two ways) to the semi-direct product $G(\mathbb{F}_{q^m}) \rtimes \langle F \rangle$. For any such extension, the restriction of its character to the coset $G(\mathbb{F}_{q^m}) \cdot F$ is a $F$-class function, and one may consider its Shintani descent.
The following is a basic fact in the theory of Shintani descent:
\vskip 10pt

\begin{prop} \label{P:shintani}
There is an extension of $R(T_m,\theta_m)$ whose associated Shintani descent is the representation $R(T,\theta)$.  
\end{prop}
\vskip 10pt

Now we can begin our study of the restriction problem for unitary groups over finite fields. 
Let 
\[  \begin{cases}
\pi_1 = R(T, \theta) \\
 \pi_2 = R(T', \theta') \end{cases} \]
 be irreducible Deligne-Lusztig representations of $\U_n(\mathbb{F}_q)$ and $\U_{n-1}(\mathbb{F}_q)$ respectively, and let  $\chi_i$ be the character
of $\pi_i$.  We shall consider the quadratic base change of $\pi_i$.
By Proposition \ref{P:shintani}, there are extensions of the representations 
\vskip 5pt
\[  \begin{cases}
\text{$R(T_2, \theta_2)$ of $\GL_n({\Bbb F}_{q^2})$;} \\
\text{ $R(T'_2, \theta'_2)$ of  $\GL_{n-1}({\Bbb F}_{q^2})$,} 
\end{cases} \] 
whose associated Shintani descents are $\chi_1$ and $\chi_2$ respectively.
Fixing such an extension in each case, we denote the corresponding character of this distinguished extension by $\chi_i'$.  
 \vskip 10pt

It follows that 

$$ 2 \cdot \langle \chi'_1,\chi'_2\rangle _{\GL_{n-1}({\Bbb F}_{q^2}) \rtimes \langle F\rangle } =  
\langle\chi'_1,\chi'_2\rangle _{\GL_{n-1}({\Bbb F}_{q^2})} +   
\langle\chi'_1,\chi'_2\rangle _{\GL_{n-1}({\Bbb F}_{q^2})\cdot F},$$  
and thus
$$2 \cdot \langle\chi'_1,\chi'_2\rangle _{\GL_{n-1}({\Bbb F}_{q^2}) \rtimes \langle F\rangle } =  
\langle \chi'_1,\chi'_2\rangle _{\GL_{n-1}({\Bbb F}_{q^2})} +   
\langle\chi_1,\chi_2\rangle _{\U_{n-1}({\Bbb F}_{q})}.$$  
Now we observe that:
\vskip 5pt

\begin{enumerate}[(i)]
\item  the left hand side of this last equality is an
{\em even} integer;
\vskip 5pt

\item  the quantity 
$\langle\chi'_1,\chi'_2\rangle _{\GL_{n-1}({\Bbb F}_{q^2})}$ was computed in Theorem \ref{T:GL-finite}, under the assumption that $R(T_2',\theta'_2)$ is an irreducible representation;
\vskip 5pt

\item the quantity  $\langle\chi_1,\chi_2\rangle _{\U_{n-1}({\Bbb F}_{q})}$ is equal to $0$ or $1$ in certain cases, by Proposition \ref{P:AGRS-finite}.
\end{enumerate}
\vskip 5pt
\noindent  Together, these observations allow one to compute $\langle\chi_1,
\chi_2\rangle _{\U_{n-1}({\Bbb F}_{q})}$ in certain situations.
\vskip 10pt

Namely, let us assume that $\pi_1$ and $\pi_2$ are irreducible Deligne-Lusztig representations, 
and suppose further that $\pi_2$ is cuspidal. Then the quadratic base change $\pi_1'$ and $\pi_2'$ of $\pi_1$ and $\pi_2$ are  are irreducible full principal series representations of $\GL_n(\mathbb{F}_{q^2})$ and $\GL_{n-1}(\mathbb{F}_{q^2})$. Thus, Theorem \ref{T:GL-finite} implies that
\[  \langle\chi'_1,\chi'_2\rangle _{\GL_{n-1}({\Bbb F}_{q^2})} = \begin{cases}
\text{$1$, 
if the cuspidal supports of $\pi_1'$ and $\pi_2'$ are disjoint,} \\
 \text{an even integer,  otherwise.} 
\end{cases} \] 
On the other hand, by Proposition \ref{P:AGRS-finite},  
$\langle\chi_1,\chi_2\rangle _{\U_{n-1}({\Bbb F}_{q})}$ is either 0 or 1. 
Therefore we get the following theorem as our only option.
\vskip 5pt

\begin{thm}
Let $\pi_1$ and $\pi_2$ be irreducible Deligne-Lusztig representations and suppose that $\pi_2$ is cuspidal. Then 
\[  \dim \Hom_{\U_{n-1}(\mathbb{F}_q)}(\pi_1, \pi_2) \ne 0 \] 
if and only if the cuspidal supports of the base change representations
 $\pi'_1$ and $\pi'_2$ are disjoint, in which case the Hom space has dimension 1.
 \end{thm}
\vskip 5pt

In particular, this theorem completes the proof of Theorem \ref{T:depth-zero-sc}.
Indeed, in the setting of Theorem \ref{T:depth-zero-sc}, we need to show that the distinguished representation $\pi_{\chi} =\pi_1 \times \pi_2$ of $\U(W) \times \U(W_0)$ satisfies
\[  \Hom_H(\pi_{\chi}, \nu) \ne 0. \]
By the argument in the proof of Proposition \ref{P:AGRS-finite}, it is sufficient to show that the representation 
\[  R(\alpha) \otimes R(\beta) \quad \text{ of $\U_{n-p}(\mathbb{F}_q) \times \U_{m-p}(\mathbb{F}_q)$} \] 
satisfies
 \[  \Hom_{H(\mathbb{F}_q)}( R(\alpha) \otimes R(\beta) , \nu) \ne 0. \]
The desired non-vanishing then follows from Proposition \ref{P:AGRS-finite2}(i) and the above theorem, using the fact that the quadratic base change of $R(\alpha)$ and $R(\beta)$ have disjoint cuspidal support.

\vskip 5pt


\vskip 15pt

\section{Langlands-Vogan packets for small unitary groups} \label{S:LV}

The rest of this paper is devoted to verifying [GGP, Conjecture 16.3] or its variant [GGP, Conjecture 20.1] in various low rank examples in the unitary and symplectic cases. In this section, we explicate the Langlands-Vogan parameterization of irreducible representations of $\U(V)$ where $V$ is a hermitian (or skew-hermitian) space over $k$ of dimension $\leq 3$.
\vskip 5pt

When $\dim_k V =1$, the group $\U(V)$ is naturally isomorphic to the subgroup $k^1$ of norm one elements in $k^{\times}$, via its scalar action on $V$. The map
\[  x \mapsto x/x^{\sigma} \]
gives an isomorphism of $k^{\times}/k_0^{\times}$ with $\U(V)$. 
The only other pure inner form of $\U(V)$ is the group $\U(V')$ where $V'$ is obtained from $V$ by scaling the hermitian form on $V$ by an element in $k_0^{\times} \smallsetminus \mathbb{N}k^{\times}$. 
\vskip 5pt

In this case, an $L$-parameter for $\U(V)$ is a 1-dimensional conjugate-orthogonal representation $M$ of $WD(k)$, which corresponds via local class field theory to a character of 
 $k^{\times}/k_0^{\times}$, and hence to characters $\chi_M$ of  $\U(V)$ and $\chi'_M$ of $\U(V')$. The Vogan packet associated to $M$ is then 
 \[  \Pi_M = \{  \chi_M, \chi_M'\}. \]
 The component group $A_M$ is $\Z/2\Z$ and the trivial character of $A_M$ corresponds to the character $\chi_M$ of $\U(V)$.
 \vskip 10pt

 Now consider the case when $\dim V  =2$. We take $V$ to be the split hermitian space, and denote the other rank 2 hermitian space (which is anisotropic) by $V'$. In this case, the groups $\U(V)$ and $\U(V')$ are closely related to the group $\GL_2$ and its inner form 
 $D^{\times}$, where $D$ is the unique quaternion division algebra over $k_0$. 
 
 \vskip 5pt
 
 More precisely, given a quaternion algebra $B$ over $k_0$ (possibly split), we fix an embedding 
\[  k \hookrightarrow B \]
of algebras over $k_0$  and regard $B$ as a 2-dimensional vector space over $k$ via left multiplication. All such embeddings of $k$ into $B$ are conjugate under $\Aut_{k_0}(B)$ by the Skolem-Noether theorem. There is an element $b \in B$ (of trace zero) which normalizes $k$ and whose conjugation action on $k$ is the involution $\sigma$; moreover, all other such elements are of the form $\lambda \cdot b$ for $\lambda \in k$.  We thus have a decomposition 
\[  B =  k \cdot 1 + k \cdot b. \]
Define a nondegenerate hermitian form on $B$ by 
\[  \langle x,y \rangle = \text{projection of $x \cdot \overline{y}$ onto $k \cdot 1$}, \]
where $y \mapsto \overline{y}$ is the canonical involution on $B$;
 let $V_B$ be the associated hermitian space.
If $B$ is split, then $V_B$ is the split hermitian space $V$, whereas if $B$ is the quaternion division algebra $D$ over $k_0$, then $V_B$ is the anisotropic hermitian space $V'$. 

\vskip 10pt

The associated unitary similitude group is given by
$$\GU(V_B) \cong (B^{\times} \times k^{\times})/\Delta k_0^{\times} $$
with an element $(b, t) \in B^{\times} \times k^{\times}$ acting on $B$ by
\[  (b,t)(x) = t x b^{-1}. \]
The similitude character is given by
\[  (b, t) \mapsto \mathbb{N}t \cdot \mathbb{N}b^{-1}, \]
so that
\[  \U(V_B) = \{ (b,t) \in \GU(V_B): \mathbb{N}b = \mathbb{N} t\}. \]
Observe that $\U(V_B)$ is a subgroup of
\[  \GU^+(V_B) = ((B^{\times})^+ \times k^{\times})/ \Delta k_0^{\times}, \]
where 
\[  (B^{\times})^+ = \{ b \in B^{\times}: \mathbb{N} b \in \mathbb{N} k^{\times} \}. \]
Indeed, it is easy to see that
\[  \GU^+(V_B)  = \U(V_B) \cdot Z_{\GU(V_B)}, \]
where
\[  Z_{\GU(V_B)} = (k_0^{\times} \times k^{\times})/\Delta k_0^{\times} \cong k^{\times} \]
is the center of $\GU(V_B)$.

 \vskip 5pt
For later purposes, we describe here a nondegenerate rank $1$ hermitian subspace of $V_B$.
Consider the  nondegenerate subspace 
\[  L_B = k \cdot b \hookrightarrow B \]
and observe that its orthogonal complement $L_B^{\perp} = k \cdot 1$ is isomorphic to  $\langle 1 \rangle$. 
The pointwise stabilizer of $L_B^{\perp}$ in $\U(B)$ is the diagonal subgroup 
\[ 
 \U(L_B) \cong k^{\times} / k_0^{\times} \stackrel{\Delta}\longrightarrow (B^{\times} \times k^{\times})/\Delta k_0^{\times}.  \]

\vskip 10pt
We now come to the representation theory of $\U(V_B)$. 
Observing that the $L$-packets of $\GU(V_B)$ are all singletons, we take
an $L$-packet of $\U(V_B)$ to be the set of irreducible constituents of the
restriction of an irreducible representation of $\GU(V_B)$ to $\U(V_B)$.
Since 
\[  \GU^+(V_B) = \U(V_B) \cdot Z_{\GU(V_B)}, \]
we see that the restriction of an irreducible representation of $\GU(V_B)$ to $\U(V_B)$ is completely determined by its restriction to $\GU^+(V_B)$.  In other words, from the representation theoretic point of view, we may work with $\GU^+(V_B)$ in place of $\U(V_B)$.
\vskip 10pt

More precisely, if $\tau \boxtimes \chi$ is an irreducible representation of 
\[ \GU(V_B) = (B^{\times} \times k^{\times})/\Delta k_0^{\times},\] 
 then its restriction to $\GU^+(V_B)$ is equal to
\[  \tau|_{(B^{\times})^+} \boxtimes \chi, \]
and it is known that $\tau|_{(B^{\times})^+}$ is either irreducible or is the sum of two inequivalent summands. Moreover, the latter holds if and only if $\tau \otimes \omega_{k/k_0} \cong \tau$, in which case we say that $\tau$ is dihedral with respect to $k/k_0$.
Then the $L$-packet of $\U(V_B)$ associated to $\tau$ is the set
\[ \Pi_{B,\tau, \chi} =  \{ (\tau^+ \boxtimes \chi)|_{\U(V_B)}: \text{$\tau^+$ is an irreducible summand of   $\tau|_{(B^{\times})^+}$} \}, \]
 which has cardinality 1 or 2. Observe that if $\mu$ is any character of $k_0^{\times}$, then 
 \[  \Pi_{B, \tau \otimes (\mu^{-1} \circ \det), \chi \cdot (\mu \circ \mathbb{N})} = \Pi_{B,\tau,\chi}. \]
 If $N$ is the $L$-parameter of $\tau$, we also write $\Pi_{B,N,\chi}$ for $\Pi_{B,\tau,\chi}$. 
  \vskip 10pt

To attach $L$-parameters to these packets, recall that
an $L$-parameter in this case is a two dimensional conjugate-symplectic representation $M$ of $WD(k)$. Now we note:
\vskip 5pt

\begin{prop} \label{P:conj-sym}
(i) Let $\tau \boxtimes \chi$ be an irreducible representation of $\GU(V)$, so that $\omega_{\tau} \cdot \chi|_{k_0^{\times}} =1$.  If $N$  is the $L$-parameter of $\tau$, then the representation
\[  M = N|_{WD(k)} \otimes  \chi \]
of $WD(k)$ is conjugate-symplectic.
\vskip 10pt

\noindent (ii) Conversely, any 2-dimensional conjugate-symplectic representation $M$ of $WD(k)$ arises in this way from an irreducible representation $\tau \boxtimes \chi$ of $\GU(V)$, which is well-defined up to twisting by $(\mu^{-1} \circ \det) \boxtimes \mu  \circ \mathbb{N}$ for some character $\mu$ of $k_0^{\times}$. 
\end{prop}

\begin{proof}
(i) Let 
\[  (-, - ) : N \otimes N \longrightarrow \wedge^2 N = \det N \]
be the natural skew-symmetric $WD(k_0)$-equivariant form.
Also, for $s \in WD(k_0) \smallsetminus WD(k)$, there is a $WD(k)$-equivariant isomorphism
\[  N^s    \longrightarrow N \]
given by the action of $s^{-1}$ on $N$. 
By composition, we obtain a $WD(k)$-equivariant bilinear form
\[ 
 N \otimes N^s \stackrel {1 \otimes s^{-1}} \longrightarrow  N \otimes N 
\longrightarrow \det N.
 \]
 Twisting $N$ with $\chi$ and $N^s$ with $\chi^s$ gives a conjugate duality
 \[  B: M \otimes M^s \longrightarrow \det N \cdot \chi \cdot \chi^s  = 1, \]
 where the last equality follows from the fact that
 \[  \det N \cdot \chi|_{k_0^{\times}} = 1 \quad \text{on $\mathbb{N}k^{\times}$}. \]
To see that this conjugate duality has sign $-1$, we write $\rho_M$ for the action of 
$WD(k)$ on $M$ and $\rho_N$ for the action of $WD(k_0)$ on $N$ and compute:
\begin{align}
&B(n, \rho_M(s^2) m)\notag \\
 = &\chi(s^2) \cdot ( n, \rho_N( s)^{-1} \cdot \rho_N(s)^2 m ) \notag \\
= &\chi(s^2) \cdot (n, \rho_N(s) m) \notag \\
= &-\chi(s^2) \cdot  (\rho_N(s)m, n) \notag \\
= &-\chi(s^2) \cdot \det N(s) \cdot  (m, \rho_N(s)^{-1}n) \notag \\
= &-\chi(s^2) \cdot \det N(s) \cdot B(m,n). \notag 
\end{align} 
 But if $t \in k_0^{\times} \smallsetminus \mathbb{N}k^{\times}$, then
 \[ \chi(s^2) \cdot \det N(s) = \chi|_{k_0^{\times}}(t) \cdot \omega_{\tau}(t) = 1. \]
 This proves (i). 
 \vskip 5pt

\noindent (ii) Conversely, if $M$ is conjugate-symplectic, then
$\det M$ is conjugate-orthogonal and thus has the form $\chi/ \chi^{\sigma}$ for some character $\chi$ of $WD(k)$. Moreover, such a $\chi$ is well-determined up to a character of the form $\mu \circ \mathbb{N}$.  The representation $M \otimes \chi^{-1}$ is then $\sigma$-invariant and hence is the restriction to $WD(k)$ of a representation $N$ of $WD(k_0)$.  Such an $N$ is not unique, as one can choose to twist any irreducible summand of $N$ by $\omega_{k/k_0}$. In any case, 
we have 
\[  M = N|_{WD(k)} \otimes \chi,\] 
and 
\[  \det N \cdot \chi|_{k_0^{\times}} = 1 \quad \text{or} \quad \omega_{k/k_0}. \]
We need to show that $N$ can be chosen so that the first possibility holds, and that this choice of $N$ is unique up to twisting by $\omega_{k/k_0}$ (for $\chi$ fixed). For this, we consider various cases:
\vskip 5pt

\begin{enumerate}[(a)]
\item if $N$ is reducible,  then by twisting one of its irreducible summand (which is 1-dimensional) 
 by $\omega_{k/k_0}$ if necessary, we can ensure that $\det N \cdot \chi|_{k_0^{\times}} =1$. 
 For a fixed choice of $\chi$, the only other $N$ for which this holds is $N \otimes \omega_{k/k_0}$.
 
 \vskip 5pt
 \item if $M$ is irreducible, then Schur's lemma implies that any two conjugate dualities of $M$ are multiples of each other and thus must have sign $-1$ in our setting. 
 On the other hand, 
the construction in the proof of (i) gives a conjugate duality on 
 $M$ which has sign 
 \[  \begin{cases}
 +1, \text{  if  $\det N \cdot \chi|_{k_0^{\times}} = \omega_{k/k_0}$;} \\
 -1, \text{  if $ \det N \cdot \chi|_{k_0^{\times}} = 1$.} 
 \end{cases} \]
 Thus, $ \det N \cdot \chi|_{k_0^{\times}} = 1$ in this case.  

\vskip 5pt

\item if $M$ is reducible but $N$ is irreducible, then we have
\[  M = M_1 + M_2 \]
with $M_1$ and $M_2$ distinct conjugate-symplectic, and
\[  M_1 \cdot M_2 = \chi/\chi^{\sigma}. \]
Moreover,
\[  N = \text{ind}_{WD(k)}^{WD(k_0)} M_1 \chi^{-1} \]
so that
\[  \det N = \chi^{-1}|_{k_0^{\times}}. \]
In particular, $ \det N \cdot \chi|_{k_0^{\times}} = 1$ in this case as well.
\end{enumerate}
 Hence (ii) is proved.
\end{proof}
\vskip 10pt

In view of the above proposition, we set the $L$-parameter 
associated to the packet $\Pi_{B, \tau,\chi}$ to be the conjugate-symplectic representation
\[  M = N|_{WD(k)} \otimes  \chi, \]
with $N$ the $L$-parameter of $\tau$.
 Given a conjugate-symplectic $M$, with associated pair $(\tau,\chi)$ as in Proposition 
 \ref{P:conj-sym}(ii), the associated Vogan packet is
\[  \Pi_M = \bigcup_B \Pi_{B, N, \chi}, \]
where the union is taken over the two quaternion algebras over $k_0$. 
 \vskip 10pt
 
\noindent{\bf Remark:} It has been shown by Konno-Konno [KK] that  the above construction of $L$-parameters agrees with the one supplied by the theory of twisted endoscopy (i.e. base change to $\GL(2)$ over $k$), which has been achieved by Rogawski [Ro]  using the stable trace formula. 
 
\vskip 10pt

The following table lists the various possibilities of $M$, $\Pi_M$ and the component group $A_M$, depending on the type of $\tau$'s. 
\vskip 10pt

\begin{center}
\begin{tabular}{|c|c|c|c|}
\hline \\
$\tau$  & $M$ & $\Pi_M$ & $A_M$  \\
\hline
& & &  \\
non-dihedral  principal series
& $P + {^\sigma}P^{\vee}$, &  1 representation &  trivial  \\

(with respect to $k/k_0$)  & $P \ncong  {^\sigma}P^{\vee}$ & on $\U(V)$ &  \\
\hline 
& & & \\
non-dihedral discrete series 
& irreducible & 1 representation on $\U(V)$  & $\Z/2\Z$ \\

(with respect to $k/k_0$)  & conjugate-symplectic  & and 1 on $\U(V')$ &  \\ 
\hline 
& & & \\
dihedral principal series 
&  $2 \cdot M'$, & 2 representations  & $\Z/2\Z$ \\
(with respect to $k/k_0$)  &  $M'$ conjugate-symplectic & on $\U(V)$ &  \\
\hline 
& & & \\
dihedral discrete series 
& $M_1+M_2$, $M_1 \ncong M_2$ & 2 representations on $U(V)$ &  $\Z/2\Z \times \Z/2\Z$ \\ 
(with respect to $k/k_0$)  &  conjugate-symplectic & and 2 on $\U(V')$  &  \\
\hline
\end{tabular}
\end{center}
\vskip 10pt

\noindent If the conjugate-symplectic representation $M$ is of the last 
two types in the above table, we shall call $M$ dihedral with respect 
to $k/k_0$. If it is of the first two type, we shall call it 
non-dihedral with respect to $k/k_0$.
\vskip 10pt

From the above table, we see that $\# \Pi_M = \# A_M = \# \text{Irr}(A_M)$. 
To index the representations in $\Pi_M$ by $\text{Irr}(A_M)$, we need to fix  a generic character of $\U(V)$.
According to [GGP, Prop. 12.1(2)],  a generic character of $\U(V)$ is specified by giving a nontrivial additive character 
\[  \psi: k/k_0 \to \Sc^1. \]
Via the description of $\GU(V)$ in terms of $\GL_2(k_0)$ given above, this then corresponds to a generic character of $\GL_2(k_0)$, which is given by an additive  character 
\[  \psi^0 : k_0 \to \Sc^1. \]
To describe the precise relation between $\psi$ and $\psi^0$, we need to start with an explicit embedding $k \to M_2(k_0)$. We do this in a standard way,  by choosing a trace zero element $e$ of $k$, which gives  
\[  k = k_0 \cdot 1+ k_0 \cdot e \quad \text{and}\quad  \text{End}_{k_0}(k) \cong M_2(k_0).\]
The multiplication action of $k$ on itself then gives an embedding $k \hookrightarrow  M_2(k_0)$. 
Moreover, the action of $\sigma$ on $k$ gives rise to an element $b$ in $\text{End}_{k_0}(k)$, so that
\[  M_2(k_0) = k + k \cdot b. \] 
After some calculations with these explicit data, which we will omit here, one sees that  $\psi$ and $\psi^0$ are related by
\[  \psi(x) = \psi^0( {\rm Tr}(e^{-1} \cdot x)) \quad \text{for all $x \in k$.} \]
We stress again that the above relation depends crucially on the choice of the trace zero element $e$, though we will not need to make use of this relation in this paper.
\vskip 10pt

In any case, with $\psi: k/k_0 \to \Sc^1$ and hence $\psi^0: k_0 \to \Sc^1$ fixed, we note that
$\tau|_{\GL_2(k_0)^+}$ has a unique $\psi^0$-generic constituent and hence the Vogan packet 
$\Pi_M$ has a unique $\psi$-generic element. 
 We then decree that 
\vskip 5pt

\begin{enumerate}[(i)]
\item  the trivial character of $A_M$ corresponds to $\psi$-generic element in $\Pi_M$;
\vskip 5pt

\item a character of $A_M$ corresponds to a representation of $\U(V)$ if and only if it is trivial on the image of the central element $-1 \in {^L}\U(V)$. 
\end{enumerate}
\vskip 10pt

From the above table, we see that these requirements completely determine the bijection 
\[  \Pi_M \leftrightarrow \text{Irr}(A_M), \]
except in the last case, where $\tau$ is a dihedral (with respect to $k/k_0$) discrete series representation
of $\GL_2(k_0)$. In that case, if $\tau'$ denotes the Jacquet-Langlands lift of $\tau$ to $D^{\times}$, then we do not know how to label the two summands of $\tau'|_{(D^{\times})^+}$ using the two characters of $A_M$ which are nontrivial on the central $-1$. 
However, in \S \ref{S:endo-theta}, we shall resolve this issue when we describe an alternative construction of these Vogan packets using theta correspondence.

\vskip 10pt

 Finally, we consider the case when $\dim V = 3$. In this case, the only other pure inner form of $\U(V)$ is the group $\U(V')$ where $V'$ is the hermitian space obtained from $V$ via scaling 
 by an element of $k_0^{\times} \smallsetminus \mathbb{N}k^{\times}$. In this case, the Vogan packets have been defined by Rogawski [Ro] via base change to $\GL(3)$ over $k$ using the stable trace formula. 
 
 \vskip 5pt
 
 The $L$-parameters are conjugate-orthogonal representations $M$ of $WD(k)$ of dimension $3$. When $M$ is irreducible, the associated Vogan packet is said to be stable; it consists of a representation of $\U(V)$ and the same representation regarded as a representation of $\U(V')$. The component  group $A_M$ is $\Z/2\Z$ and we decree that the trivial character corresponds to the representation of $\U(V)$. 
 On the other hand, when $M$ is reducible, the associated Vogan packet is said to be endoscopic. In \S \ref{S:endo-theta}, we shall describe a construction of the endoscopic packets, and the labelling of their representations by $\text{Irr}(A_M)$, via the approach of theta correspondence.
 \vskip 15pt
 
 \section{Theta correspondence}  \label{S:theta}
 
 The goal of this section is to review the necessary background and framework for the theta correspondence for unitary groups. This is necessary for the construction of endoscopic Vogan packets
 of $\U(2)$ and $\U(3)$ which will be given in the following section.
 \vskip 10pt

  Let $V$ be a hermitian space and $W$ a skew-hermitian space over $k$. To consider the theta
 correspondence for the dual pair $\U(V) \times \U(W)$, one requires certain additional data:
 \vskip 5pt
 
 \begin{enumerate}[(i)]
 \item an additive character $\psi_0 : k_0 \to \Sc^1$; 
 \vskip 5pt
 
 \item a character $\mu: k^{\times} \to \CC^{\times}$ such that $\mu|_{k_0^{\times}} = \omega_{k/k_0}$.
 \end{enumerate}
 \vskip 5pt

 \vskip 5pt

 To elaborate,  the tensor product $\text{Res}_{k/k_0} (V \otimes_k W)$ has a natural symplectic form defined by 
 \[  \langle v_1 \otimes w_1,  v_2 \otimes w_2 \rangle =  \text{Tr}_{k/k_0}(\langle v_1,v_2 \rangle_V \cdot \langle w_1, w_2 \rangle_W). \]
 Note that many authors (for example [HKS]) include a factor $1/2$ on the right hand side, but we shall not follow this convention here. 
 In any case,  there is a natural map
 \[ i:  \U(V) \times \U(W) \longrightarrow \Sp(V \otimes W/k_0). \]
 One has the metaplectic $\Sc^1$-cover $\text{Mp}(V \otimes W)$ of $\Sp(V \otimes W)$, and the character $\psi_0$ (together with the form $\langle-,-\rangle$ on $V \otimes W$) determines a Weil representation $\omega_{\psi_0}$ of   
 $\text{Mp}(V \otimes W)$. 
 To obtain a representation of $\U(V) \times \U(W)$ from  $\omega_{\psi_0}$, however,  one needs to specify a splitting of the map $i$ to the metaplectic cover. This is quite subtle,  but was completely understood by Gelbart-Rogawski [GRO], Kudla [K] and Harris-Kudla-Sweet [HKS]; it requires the additional data above.  
 \vskip 10pt
 
 More precisely,  the data $(V,\psi_0, \mu)$ determines a splitting 
 \[  i_{V, \mu, \psi_0}: \U(W) \hookrightarrow {\rm Mp}(V \otimes W), \]
 whereas the data $(W,\psi_0,\mu)$ determines a splitting
 \[  i_{W,\mu,\psi_0} : \U(V) \hookrightarrow {\rm Mp}(V \otimes W) \]
 whose image commutes with that of $i_{V,\mu,\psi_0}$.
 In [HKS], the above splittings can be defined for the choice  of any pair of characters $(\chi, \chi')$ of $k^{\times}$ satisfying
 \[  \chi|_{k_0^{\times}} = \omega_{k/k_0}^{\dim V} \quad \text{and} \quad \chi'|_{k_0^{\times}} = \omega_{k/k_0}^{\dim W}. \]
In their terminology, our splittings are relative to the pair of characters
\[  \chi = \mu^{\dim V} \quad \text{and} \quad \chi' = \mu^{\dim W}. \]  
In particular, by [HKS, Corollary A.8],  a property of this splitting is that the images of the centers of $\U(V)$ and $\U(W)$ are identified, so that the theta correspondence we consider here preserves central characters.  
 \vskip 5pt
 
 Using the above splittings, one obtains a Weil representation
 \[  \omega_{\psi_0, \mu} = \omega_{\psi_0} \circ (i_{W,\mu, \psi_0, \delta} \times i_{V, \mu, \psi_0}) \]
 of $\U(V) \times \U(W)$.   
  Moreover, the Weil representation $\omega_{\psi_0,\mu}$ depends only on the orbit of $\psi_0$ under $\mathbb{N}k^{\times}$.
 Thus, given an irreducible  representation $\pi$ of $\U(W)$, we have its big and small theta lift 
 $\Theta_{\psi_0, \mu}(\pi)$ and $\theta_{\psi_0,\mu}(\pi)$ on $\U(V)$.
 \vskip 10pt
 
 It would appear that, by restricting $(\chi_1, \chi_2)$ (as in [HKS]) to have the special form taken here, we are losing one degree of freedom. However, this lost degree of freedom can be regained by allowing twisting of the theta lifts by 1-dimensional characters of $\U(V)$, i.e. if we consider 
 $\theta_{\psi_0, \mu}(\pi) \otimes (\chi \circ \det)$ as well.  
 \vskip 10pt

 It is also useful to consider the theta correspondence for similitude groups. Let 
 \[  R  \subset \GU(V) \times \GU(W) \]
 be the subgroup consisting of elements $(g,h)$ such that $\text{sim}(g) \cdot \text{sim}(h) = 1$. 
 Then the Weil representation $\omega_{\psi_0,\mu}$ has a natural extension to $R$.
 Now observe that 
 \[  R \subset \GU^+(V) \times \GU^+(W) \]
 where $\GU^+(V)$ consists of those elements $g \in \GU(V)$ such that $\text{sim}(g)$ lies in the image of the similitude map of $\GU(W)$, and analogously for $\GU^+(W)$. Then   
 one may consider the induced representation
 \[  \Omega_{\psi_0,\mu} = {\rm ind}_R^{\GU^+(V) \times \GU^+(W)} \omega_{\psi,\mu} \]
 of $\GU^+(V) \times \GU^+(W)$, which depends only on the orbit of $\psi_0$ under $\mathbb{N}k^{\times}$ (and may even be independent of $\psi_0$ in some cases). We can now consider
 the theta correspondence for $\GU^+(V) \times \GU^+(W)$ associated to $\Omega_{\psi_0,\mu}$. In particular, for a representation $\pi$ of $\GU^+(W)$, we have its big and small theta lifts $\Theta_{\psi,\mu}(\pi)$ and $\theta_{\psi, \mu}(\pi)$ on $\GU^+(V)$. 
 \vskip 10pt

In this paper, we will be considering the theta correspondence for $\U(V) \times \U(W)$ with 
$|\dim V - \dim W| \leq 1$. In this case, there are some rather precise conjectures about the behavior of the theta correspondence in the literature (see for example [HKS, \S 7] and [P5]). We formulate these as the following working hypothesis.
 \vskip 15pt
 
 \noindent{\underline{\bf  Working hypothesis}}: Let $V$ be a hermitian space and let $W$ be a  skew-hermitian space, and consider the theta correspondence for $\U(V) \times \U(W)$ relative to the data $(\psi_0,\mu)$. For an irreducible representation $\pi$ of
$\U(V)$, let $\theta_{\psi_0,\mu}(\pi)$ denote the (small) theta lift of $\pi$ to $\U(W)$.
\vskip 10pt

\begin{enumerate}[(a)]
\item If $\dim V = \dim W$, then the Langlands parameters of $\pi$ and $\theta_{\psi_0,\mu}(\pi)$ are the same (if the latter is nonzero). For a given $L$-parameter $M$, the theta correspondence induces a permutation of the Vogan packet $\Pi_M$ to itself. 
This bijection is given by translation by a character of the component group $A_M$,  as given in [P5].
\vskip 10pt

\item  If $\dim V= \dim W -1$, then the  Langlands parameters $M$ of $\pi$ and
$N$ of $\theta_{\psi_0,\mu}(\pi)$ are related to each other by:
\[
N = \mu^{-1} M + \mu^{\dim V}. \]
The theta correspondence relative to $(\psi_0, \mu)$ gives an injection
\[  \theta_{\psi_0,\mu, V, W}: \Pi_{V,M} \hookrightarrow  \Pi_{W,N}. \]
This injection can be naturally described in terms of the characters of the component groups of $M$ and $N$ as follows. Assume for simplicity that $\mu^{\dim V}$ does not occur in $\mu^{-1}M$, so that $A_N = \Z/2\Z \times A_M$. For an appropriately normalized Langlands-Vogan parametrization, the above injection is described by the natural map
\[  \text{Irr}(A_M) \longrightarrow \text{Irr}(A_N) = \{ \pm 1\} \times \text{Irr}(A_M)  \]
given by
\[  \rho \mapsto (\epsilon, \rho)  \]
where the sign $\epsilon$ is completely determined by $\rho$ and the space $W$.
\vskip 5pt

Moreover, as $V$ and $W$ vary over all hermitian and skew-hermitian spaces of the specified dimensions, one has 
\[  \Pi_N = \bigcup_{V, W} \theta_{\psi_0,\mu,V, W}(\Pi_{V,M}), \]
where the union is disjoint and we ignore the theta lifts which are zero.  The disjointness of the union is in fact a consequence of the main result of [HKS] on theta dichotomy.   
 \end{enumerate}
 \vskip 15pt
 
 The above working hypothesis can be made more precise, especially its relation with the Langlands-Vogan parameterization.  In the following, we shall consider the low rank cases, with 
$\dim V \leq 2$ and $\dim W  \leq 3$.
In these cases, we shall verify the above working hypothesis in its precise form. We note that these low rank cases are the  only ones in which the Langlands-Vogan parameterization is fully understood for $\U(V)$ and $\U(W)$.
\vskip 5pt

For example, statement (a) for $\dim V =1$ is a result of Moen [Mo], Rogawski [Ro2] and Harris-Kudla-Sweet [HKS] (see Theorem \ref{T:HKS} below), whereas the case when $\dim V=2$ is verified in Theorem \ref{T:theta-U2} below. On the other hand, statement (b) for $\dim V = 1$ is easy to check, and the case of  $\dim V = 2$ is due to Gelbart-Rogawski-Soudry [GRS].
 
 \vskip 15pt
 
 \section{Endoscopic packets and theta correspondence} \label{S:endo-theta}
 
 The goal of this section is to describe an alternative construction of the endoscopic packets of the unitary group $\U(V)$, via theta correspondence,  when $\dim V = 2$ or $3$.
 We shall rely heavily on the framework and notation of the previous two sections.
 \vskip 10pt

Our first case of interest is the theta correspondence for a skew-hermitian space $W$ and a hermitian space $V$ with
\[  \dim W = 1 \quad \text{and} \quad \dim V = 2. \]
 We shall use the associated theta correspondence to construct certain Vogan packets on $\U(V)$. 
We shall fix once and for all a trace zero element $\delta \in k^{\times}$, so that
\[  \delta^{\sigma} = - \delta. \]

\vskip 5pt
Recall that
in \S \ref{S:LV}, we have given a construction of the rank $2$ hermitian spaces $V_B$ in terms of quaternion algebras $B$ over $k_0$.  Suppose that
 \[  M = M_1 + M_2 \]
 is a 2-dimensional conjugate-symplectic representation of $WD(k)$, with $M_i$ conjugate-symplectic (but not necessarily distinct).  As we explained in the previous section, such an $M$ gives rise to a Vogan packet $\Pi_M$ of $\U(V_B)$. If we fix an additive character 
 \[  \psi: k/k_0 \longrightarrow \Sc^1 \]
 then there should be an associated bijection
 \[  J(\psi): \Pi_M \longleftrightarrow \text{Irr}(A_M). \]
 It is the Vogan packet $\Pi_M$, together with the bijection $J(\psi)$, that we would like to construct using theta correspondence.
 In fact, since the Vogan packets on $\U(V_B)$ are defined by
 restriction from $\GU(V_B)$, it will be better to consider the theta correspondence for the similitude groups $\GU(W) \times \GU^+(V_B) $, with
 \[  \GU(W) \cong k^{\times}  \quad \text{and} \quad \GU^+(V_B) =((B^{\times})^+ \times k^{\times})/k_0^{\times}. \]

\vskip 5pt

To set up the theta correspondence, we need to fix the data $\psi_0$ and $\mu$. 
Since we have fixed the trace zero element $\delta \in k$, there is a unique
additive character $\psi_0$ of $k_0$ such that
\[  \psi(x) = \psi_0( \frac{1}{2} \cdot {\rm Tr}(\delta x)). \]
It follows that for any trace zero element $x \in k$,
\[  \psi(x) = \psi_0(\delta x). \]
In the rest of the paper, we shall assume that $\psi$ and $\psi_0$ are related as above, via the element $\delta$. 
\vskip 10pt

We set
\[  W = \text{the rank $1$ skew-hermitian space with discriminant $\delta$,} \]
and let $W'$ be the other rank $1$ skew-hermitian space. For any $a \in k_0^{\times}$, we let $W_a$ denote the rank $1$ skew-hermitian space obtained from $W$ by scaling by $a$. 
Finally, with $M = M_1 + M_2$ as above, we set
  \[   \mu = M_1, \]
  and let $\chi$ be any character of $k^{\times}$ such that
  \[  \chi/ \chi^{\sigma} = M_1 \cdot M_2. \]
  This is possible since $M_1 \cdot M_2$ is a character of $k^{\times}/ k_0^{\times}$. The choice of $\chi$ is not unique but any two choices differ by a character of $k^{\times}$ which is $\sigma$-invariant, or equivalently that factors through the norm map to $k_0^{\times}$.
In any case, we have
\[ M = \mu + \chi/\chi^{\sigma} \cdot \mu^{-1}, \]
and the packet $\Pi_M$ is obtained by the restriction of $\tau \boxtimes \chi$, where $\tau$ is the 
representation of $B^{\times}$ with $L$-parameter
 \[  N = \text{Ind}_{WD(k)}^{WD(k_0)} \mu\chi^{-1}. \]

\vskip 10pt

Now we may consider the theta correspondence associated to 
 the Weil representation $\Omega_{\psi_0, \mu}$ of $\GU(W_a) \times \GU^+(V_B)$. 
 Regarding $\chi$ as a character of $\GU(W_a)$, we have the theta lift
  \[  \Theta_{W_a,V_B, \psi_0, \mu}(\chi) = \theta_{W_a,V_B, \psi_0, \mu}(\chi) \] 
on $\GU^+(V_B)$.  
With $B^{\times} = \GL_2(k_0)$, the character $\psi$ determines a generic character of $\GU^+(V_B)$. We  let  $\tau^+$ be the constituent of $\tau|_{\GL_2(k_0)^+}$ such that the representation $\tau^+ \boxtimes \chi$ of $\GU^+(V_B)$ is  $\psi$-generic, and let $\tau^-$ denote the other constituent. We also let $\tau'$ be the Jacquet-Langlands lift of $\tau$ to $D^{\times}$, if it exists.
 
 \vskip 5pt

 With these notations, we have: 

  \vskip 5pt

  \begin{prop}  \label{P:endo-U2}
If $B$ is split, so that $V_B = V$, then
 \[ \begin{cases}
  \theta_{\psi_0, \mu, W, V} (\chi) = \tau^+  \boxtimes \chi,  \\
   \theta_{\psi_0, \mu, W', V} (\chi) = \tau^-  \boxtimes \chi. 
   \end{cases} \]
 If $B$ is non-split, so that $V_B = V'$, then
 \[  \theta_{\psi_0,\mu, W, V'} (\chi) + \theta_{\psi,\mu,W', V'} (\chi)= \tau' \boxtimes \chi, \]
 where the RHS is interpreted as $0$ if $\tau'$ does not exist.
 In particular, upon restriction to $\U(V)$ or $\U(V')$, the set
  \[  \{  \theta_{\psi_0,\mu,V, W}(\chi), \theta_{\psi_0,\mu,V, W}(\chi), \theta_{\psi_0,\mu,V, W'}(\chi), \theta_{\psi_0,\mu,V, W'}(\chi) \} \]
 is the Vogan packet $\Pi_M$ associated to the $L$-parameter 
 \[ M= M_1 + M_2 =\mu + \mu^{-1} \chi/\chi^{\sigma}. \] 
 
 \end{prop}
\vskip 10pt

Using the above construction of endoscopic packets of $\U(V)$, we can define the bijection
\[  J(\psi): \Pi_M \longleftrightarrow \text{Irr}(A_M),\]
as follows. Consider the case when $M_1 \ne M_2$, so that $A_M = \Z/2\Z \times \Z/2\Z$; this is the only case where the bijection $\Pi_M \leftrightarrow \text{Irr}(A_M)$ has some ambiguity.  We set
\[  \begin{cases}
\pi^{++}  =   \theta_{\psi_0,\mu,V, W}(\chi)  \\
\pi^{--} =   \theta_{\psi_0,\mu,V, W'}(\chi) \\
\pi^{+-} =  \theta_{\psi_0,\mu,V', W'}(\chi) \\
\pi^{-+} = \theta_{\psi_0,\mu,V', W}(\chi). \end{cases} \]
In other words, the recipe for labelling is that  
\[  \pi^{\epsilon_1, \epsilon_2}  = \theta_{\psi, \mu, W_a, V_B}(\chi) \]
where
\[  \epsilon_1 \cdot \epsilon_2 = \epsilon(B) = \begin{cases} 
1 \text{ if $B$ is split;} \\
-1, \text{  if $B$ is not split,} \end{cases}  \]
and
\[  \epsilon_2 = 
\omega_{k/k_0}(a). \]
Equivalently, if $\eta$ is a a character of $A_M$, then 
\[  \pi_{\eta}   = \theta_{\psi, \mu, W_a, V_B}(\chi) \]
if and only if 
\[ \begin{cases}
 \eta(a_1) = \epsilon(B) \cdot \omega_{k/k_0}(a), \\
\eta(a_2) = \omega_{k/k_0}(a). 
\end{cases} \]
\vskip 10pt

 We leave it to the reader to verify that under this system of bijections $J(\psi)$, the various desiderata of the Vogan parameterization listed in [GGP, \S 9 and \S 10] are satisfied. In particular, 
 the trivial character of $A_M$ corresponds to the unique $\psi$-generic representation of the packet, and if $\psi'$ belongs to the other $\mathbb{N}k^{\times}$-orbit, then the unique $\psi'$-generic representation corresponds to the character
 \[ \eta_0 (a_i) = (-1)^{\dim M_i}. \]
 Indeed, when $M$ is irreducible, $\eta_0$ is trivial, whereas when $M = M_1+M_2$ is reducible, then $\eta_0$ is the character $(--)$ of $A_M = \Z/2\Z \times \Z/2\Z$. 
 \vskip 15pt
 
It will be useful to convert the above classification into the setting of rank $2$ skew-hermitian spaces. Using  the trace zero element $\delta$, let $W_{B,\delta}$ be the skew-hermitian space obtained from $V_B$ by scaling by $\delta$, and we shall frequently write $W_B$ for $W_{B,\delta}$. Then we have
 \[  \GU(W_B) = \GU(V_B) \] 
 as subsets of $\text{End}_k(B)$.  Moreover, the notions of $L$-parameters and $L$-packets are the same for $\U(V_B)$ and $\U(W_{B,\delta})$. The only difference lies in the data needed to specify a bijection of a Vogan packet with the set of characters of the component group. In the case of $V_B$, we used an additive character 
 \[ \psi: k/k_0 \longrightarrow \Sc^1, \]
 whereas for the case of $W_{B,\delta}$, one needs an additive character of $k_0$. 
 However, it is easy to check that if a representation $\pi$ of $\U(V_B)$ is generic with respect to $\psi$, then regarded as a representation of $\U(W_{B,\delta})$, $\pi$ is generic with respect to the character given by
 \[  \psi_{00}(y) =   \psi(\delta^{-1} y), \qquad \text{for $y \in k_0$,} \]
 or equivalently
 \[  \psi(x) = \psi_{00}(\delta x) = \psi_{00}(\frac{1}{2} {\rm Tr}(\delta x)) \quad \text{for trace zero $x \in k$.} \]
 In other words, the character $\psi_{00}$ is precisely the character $\psi_0$ which we have fixed, and the bijection
 \[  J(\psi): \Pi_M \longleftrightarrow \text{Irr}(A_M) \]
 for $\U(V_B)$ is the bijection $J(\psi_0)$ for $\U(W_{B,\delta})$.
 For a character $\eta$ of $A_M$, we then have
 \[  \pi_{\eta} =  \theta_{\psi_0, \mu, W_{B,\delta}, V_a}(\chi) \]
 where $\mu$ and $\chi$ are obtained from $M$ as before, $V_a$ is the rank $1$ hermitian space with discriminant $a$, and
 \[  \begin{cases}
 \eta(a_1) = \epsilon(B) \cdot \omega_{k/k_0}(a) \\
 \eta(a_2) = \omega_{k/k_0}(a). 
 \end{cases} \]

 \vskip 15pt
 
 Finally, we consider the endoscopic Vogan packets of $\U(V)$ when $\dim V = 3$. 
Hence, we fix a rank 3 hermitian space $V$ and let $V'$ denote the other rank $3$ hermitian space. More generally, for any $a \in k_0^{\times}$, we let $V_a$ denote the hermitian space obtained from $V$ by scaling the hermitian form by $a$. 

\vskip 10pt

Consider $L$-parameters of $\U(V)$ of the form
 \[  M = M_1 + M_2 \]
 where $M_i$ are conjugate-orthogonal representations of 
$WD(k)$ of dimension $i$.  
 Unless $M \cong 3 \cdot M_1$, we may further assume that $M_1$ is distinct from any irreducible constituent of $M_2$.  It was shown in [GRS] that the Vogan packet $\Pi_M$ can by constructed using theta correspondence from $\U(W_B)$, where $W_B$ is the rank $2$ skew-hermitian space introduced above, together with twisting by 1-dimensional characters of $\U(V)$.

\vskip 10pt
 To specify the data needed for theta correspondence, note that
  $M_1$ is a character of $k^{\times}/k_0^{\times}$ and thus 
 we can hope to find a conjugate-symplectic character $\mu$ of $k^{\times}$ such that
 \[  M_1 = \mu^2. \]
 This cannot always be achieved, but if we allow ourselves to replace $M$ by a twist, then it can certainly be done. Indeed, one would simply pick a conjugate-symplectic $\mu$ and twist $M$ by $\mu^2 M_1^{-1}$. Since the Langlands-Vogan parametrization  is  compatible with twisting, there is no loss of generality in assuming that $M_1$ has a square root $\mu$ which is conjugate-symplectic. 
 \vskip 10pt

Now set
 \[  N = M_2 \cdot \mu \]
 so that $N$ is conjugate-symplectic and is an $L$-parameter for $\U(W_B)$, and 
 \[  M = \mu^2 + N \cdot \mu^{-1}. \]
For the additive character $\psi_0$ of $k_0$, we can now consider the theta correspondence associated to the Weil representation $\Omega_{\psi_0, \mu, W_B, V_a}$.
 
 \vskip 10pt
 
 More precisely, let $\Pi_N$ be the Vogan packet associated to $N$, together with the bijection
 \[  J(\psi_0): \Pi_N \longleftrightarrow \text{Irr}(A_N)\]
associated to the additive character $\psi_0$.
Then for $\eta \in \text{Irr}(A_N)$, we may consider the theta lift
\[  \theta_{\psi_{0}, \mu, W_B, V_a}(\pi_\eta), \]
where $\pi_{\eta} \in \Pi_N$  is the representation of $\U(W_B)$ (this uniquely specifies $B$) indexed by $\eta$ under $J(\psi_0)$.
As the element $a$ varies over the two representatives of $k_0^{\times}/\mathbb{N}k^{\times}$, and the character $\eta$ varies over $\text{Irr}(A_N)$, we obtain a collection of $2 \cdot \# \Pi_N$ representations (possibly zero).  It was shown by Gelbart-Rogawski-Soudry [GRS] that this set of representations so obtained is the Vogan packet associated to the endoscopic parameter $\Pi_M$. 
\vskip 5pt

The following lemma, which was shown in [GRS],  addresses more precisely the issue of non-vanishing of these theta lifts.

\vskip 5pt

\begin{lemma}
Let $M = M_1+ M_2 = \mu^2 + N \cdot \mu^{-1}$ as above. If $M \ncong 3 M_1$, assume without loss of generality that $M_1$ is distinct from any irreducible constituent of $M_2$. 
\vskip 5pt

\noindent (i) If $M \ncong 3 M_1$,  then the representation 
 $\theta_{\psi_{0}, \mu, W_B, V_a}(\pi_\eta)$ is always nonzero.
 \vskip 10pt
 
 \noindent (ii) If $M = 3 M_1$, then $N = 2 \cdot \mu^3$ and $A_N \cong \Z/2\Z$, so that we may regard $\eta = \pm 1$, depending on whether $\eta$ is trivial or not.   The representation $\theta_{\psi_{0}, \mu, W_B, V_a}(\pi_\eta)$ is nonzero if and only if
 \[  \omega_{k/k_0}( \disc V_a)  = \eta. \]
 \vskip 10pt
 
 \noindent In each case above, the non-zero representations are mutually distinct.
  Moreover, the representation $\theta_{\psi_{0}, \mu, W_B, V_a}(\pi_\eta)$ is generic if and only if 
  $\pi_{\eta}$ is generic with respect to $\psi_{0, \disc(V_a)}$.
  \end{lemma}

\vskip 10pt

We may now define a labeling of the elements in $\Pi_M$ by the irreducible characters of $A_M$.
\vskip 5pt

\begin{enumerate}[(i)]
\item If $M \ncong 3 M_1$, and $M_1$ does not occur in $M_2$, then 
\[  A_M  = A_{M_1} \times A_{M_2} = A_{M_1} \times A_N. \]
For a character $\chi = (\epsilon, \eta) \in \text{Irr}(A_{M_1}) \times \text{Irr}(A_N)$, we set
 \[  \pi^{\chi} = \pi^{\epsilon, \eta} = \theta_{\psi_0, \mu, W_B, V_a} (\pi_{\eta_V}) \]
 with
 \[  \epsilon \cdot \eta(-1) = \omega_{k/k_0}(a), \]
 and 
 \[  \eta_V = \begin{cases}
 \eta, \text{  if $\omega_{k/k_0}(\disc V)=1$; } \\
 \eta \cdot \eta_{N,0}, \text{  if $\omega_{k/k_0}(\disc V)= -1$,}
  \end{cases} \] 
 where $\eta_{N,0}$ is the character of $A_N$ which indexes the $\psi'$-generic element of $\Pi_N$.  
 More simply, when $\disc (V) = 1$, we have
 \[  \chi(a_1) = \omega_{k/k_0}(a) \cdot \eta(-1) = \omega_{k/k_0}(a) \cdot \epsilon(B) \]
 \and
 \[  \chi|_{A_{M_2}} = \eta. \]
In particular, for a character $\chi$ of $A_M= A_{M_1} \times A_N$, $\pi_{\chi}$ is a representation of $\U(V)$ if and only if $\chi(-1,-1) = 1$.
 \vskip 10pt
 
 \item If $M = 3 M_1= 3 \mu^2$, then 
 \[  A_M \cong A_N =  \Z/2\Z. \]
 For a character $\eta = \pm$ of $A_M$, we set
 \[ \pi^{\eta} = \theta_{\psi_0, \mu, W_B, V_{a}}(\pi_{\eta \cdot \omega_{k/k_0}(\disc V)}) \]
 with
 \[  \omega_{k/k_0}(a) =  \eta. \]

By part (ii) of the above lemma, this condition ensures that the theta lift above is nonzero.
 In particular, the trivial character of $A_M$ corresponds to a representation of $\U(V)$ whereas the nontrivial character corresponds to the same representation regarded on $\U(V')$.
 \vskip 10pt

\end{enumerate}

 Note that since $\dim V = 3$, there is only one orbit of generic characters for $\U(V)$, and hence 
the Vogan parameterization in this case is canonical. So it is instructive to observe that the above parameterization is independent of the choice of $\psi_0$. We leave this to the reader, as well as the verification that the above definition satisfies the desiderata of the Vogan parameterization listed in [GGP, \S 9 and \S 10].

 \section{Skew-hermitian case: $\U(1) \times \U(1)$} \label{S:U(1)xU(1)}
 
Having explicated the Langlands-Vogan parameterization of the unitary groups $\U(V)$ with $\dim V \leq 3$, we are now in a position to verify instances of [GGP, Conjecture 16.3]. 
\vskip 5pt

In this section, we consider the case of a pair of skew-hermitian spaces $W \subseteq V$, with $\dim W = \dim V =1$. Without loss of generality, we assume that
\[  \delta = \text{the discriminant of $V$}. \]
As in the previous section, 
we let  $W'  = V'$ be the other rank 1 skew-hermitian space, and more generally, for any $a \in k_0^{\times}$, we let $W_a$ be the skew-hermitian space obtained from $W$ by scaling by $a$.

 \vskip 5pt

To specify the restriction problem in this setting,
fix an additive character 
\[  \psi_0: k_0 \to \Sc^1, \]
 and a character $\mu$ of $k^{\times}$ with 
\[  \mu|_{k_0^{\times}} = \omega_{k/k_0}. \]
These determine a Weil representation $\omega_{W_a, \psi_0, \mu}$  of $\U(W_a)$.
 Also, fix two conjugate-orthogonal representations $M$ and $N$ of dimension $1$, which gives rise to characters $\alpha \times \beta$  of $\U(W_a) \times \U(W_a)$.  We are interested in determining
\[  \Hom_{\U(W_a)}( \alpha \cdot \beta , \omega_{W_a,\psi_0,\mu}).  \] This question has been resolved by Moen [Mo], Rogawski [Ro2] and Harris-Kudla-Sweet [HKS], and we state the result from [HKS, Corollary 8.5] as:
 \vskip 10pt
 
 \begin{thm} \label{T:HKS}
 For each $a \in k_0^{\times}$, let $W_a$ be the rank 1 hermitian space with discriminant $a \cdot \delta$, and
 for each $b \in k_0^{\times}$, let $V_b$ be the rank 1 hermitian space with discriminant $b$.
Given a character $\eta$ of $k^{\times}/k_0^{\times}$, which can be regarded as a character of $\U(W_a)$,  we have
 \[   \Hom_{\U(W_a)}( \eta, \omega_{W_a,V_b, \psi_0,\mu}) \ne 0 \Longleftrightarrow \epsilon(\eta \cdot \mu^{-1}, \psi_0({\rm Tr}(\delta  -))) = \omega_{k/k_0}(a \cdot b). \]
\end{thm}
 \vskip 5pt
 
 \noindent{\bf Remark:} We note that our convention here differs from [HKS] in two aspects. 
 Namely, we have adopted the convention that on $W_a \otimes V_b$,  the symplectic form is 
 ${\rm Tr}(\langle-,-\rangle_{W_a} \otimes \langle-,-\rangle_{V_b})$. In [HKS], the symplectic form is 
 \[  \frac{1}{2} \cdot {\rm Tr}(\langle-,-\rangle^{\sigma} _{W_a} \otimes \langle-,-\rangle_{V_b}). \]
 Besides the factor of $1/2$, the skew-hermitian form on $W_a$ is conjugated by $\sigma$, which is necessitated by the convention adopted by  [HKS] that skew-hermitian forms are linear in the second variable and hermitian forms are linear in the first variable.
  Conjugating the form on $W_a$ by $\sigma$ has the effect of replacing $\delta$ by $-\delta$ in [HKS, Corollary 8.5].
 \vskip 10pt

To apply the above theorem  to [GGP, Conjecture 16.3], set $\eta = \alpha \cdot \beta$ in the theorem, and note that the distinguished character $\chi_0$ of $A_M \times A_N = \Z/2\Z \times \Z/2\Z$ given in [GGP, Conjecture 16.3] satisfies
\[  \chi_0(-1,1) = \chi_0(1,-1) = \epsilon(M \otimes N(\mu^{-1}), \psi_0({\rm Tr}(\delta -)). \]
Thus, Theorem \ref{T:HKS} implies that
\[  \text{$\chi_0$ is trivial} \Longleftrightarrow \Hom_{\U(W)}( \alpha \cdot \beta, \omega_{W,\psi_0,\mu})\ne 0 
\]
and
\[  \text{$\chi_0$ is nontrivial} \Longleftrightarrow \Hom_{\U(W)}( \alpha' \cdot \beta' , \omega_{W', \psi_0, \mu}) \ne 0. \]  
\vskip 5pt
\noindent  This verifies [GGP, Conjecture 16.3] for this case.
\vskip 15pt

 \section{Restriction from $\U(2)$ to $\U(1)$} \label{S:U(2)xU(1)}

 In this section, we consider the restriction problem from $\U(2)$ to $\U(1)$. This problem has been studied by H. Saito [Sa2] and T. Konno [Ko], but we shall give an independent treatment here and relate  the result to [GGP, Conjecture 16.3].
 \vskip 10pt

Recall that in \S \ref{S:LV}, we have given a construction of rank $2$ hermitian spaces $V_B$ using quaternion algebras $B$ over $k_0$, together with a non-degenerate rank $1$ subspace:
\[  L_B \hookrightarrow V_B,\]
such that
\[  L_B^{\perp} = \langle 1\rangle. \]
When $B$ is split, this gives a pair of split hermitian spaces $L \subset V$, with
\[  \disc(L) = -1. \]
On the other hand, if $B$ is the quaternion division algebra $D$, one obtains a relevant pair 
$L' \subset V'$ with $V'$ anisotropic.
The groups 
\[  G = G(V) \times G(L) \quad \text{and} \quad G' = G(V') \times G(L') \]
are relevant pure inner forms of each other.
  \vskip 10pt

Suppose that $M$ is a conjugate-symplectic 2-dimensional representation of $WD(k)$, with component group $A_M$, so that $M$  
determines a Vogan packet $\Pi_M$ of $\U(V)$. 
We fix an additive character $\psi$ of $k/k_0$, and  translate it by $-2 \cdot \disc (L) =2$, using the resulting character  $\psi_2$ to fix the bijection
\[  J(\psi_2) : \Pi_M \leftrightarrow \text{Irr}(A_M). \]
Recall that the parameter $M$ gives rise to a representation 
\[  \text{$\tau \boxtimes \chi$ of $\GU(V) = (\GL_2(k_0) \times k^{\times}) /\Delta k_0^{\times}$}. \]
If $\tau'$ denotes the Jacquet-Langlands lift of $\tau$ to $D^{\times}$, then $\Pi_M$ is simply the set of irreducible constituents in the restriction of $\tau \boxtimes \chi$ and $\tau' \boxtimes \chi$ to $\U(V)$ and $\U(V')$ respectively.
\vskip 5pt

Similarly, suppose that $N$ is a 1-dimensional conjugate-orthogonal representation of $WD(k)$. Then $N$ determines a character $\eta$ of $k^{\times}/k_0^{\times}$, which may be regarded as a character of $\U(L_B)$. We are interested in determining 
\[  \Hom_{\U(L_B)}(\pi_B \otimes \eta, \CC) \]  
for $\pi_B \in \Pi_{M,B}$ and $\eta$ the character of $\U(L_B)$ corresponding to $N$.
 \vskip 10pt

Since the embedding 
\[  \U(L_B) \hookrightarrow \U(V_B) \subset GU^+(V_B) \]
is given by the diagonal map
\[  k^{\times}/k_0^{\times} \hookrightarrow (B^{\times} \times k^{\times})/\Delta k_0^{\times}, \]
we see that
\[ \bigoplus_B  \bigoplus_{\pi_B \in \Pi_{M,B}} \Hom_{\U(L_B)}(\pi_B \otimes \eta, \CC) \]
\[  
= \Hom_{k^{\times}}(\tau,  \chi^{-1} \eta^{-1}) +  
 \Hom_{k^{\times}}(\tau',  \chi^{-1} \eta^{-1}). \]
 \vskip 10pt

 \noindent Now we note the following theorem of Waldspurger [Wa2], Tunnell [Tu] and Saito [Sa]:
 \vskip 5pt
 
 \begin{thm} \label{T:saito-tunnell}
 Let $\tau$ be a representation of $\GL_2(k_0)$ with $L$-parameter $N(\tau)$ and Jacquet-Langlands lift $\tau'$ on $D^{\times}$. For any character $\nu$ of $k^{\times}$, with $\nu|_{k_0^{\times}} = \omega_{\tau}$, we have 
 \[  \dim \Hom_{k^{\times}}(\tau, \nu) + \dim \Hom_{k^{\times}}(\tau', \nu) =1. \]
 Moreover,  
 \[  \Hom_{k^{\times}}(\tau,  \nu) \ne 0 \Longleftrightarrow  
 \epsilon( N(\tau)|_{WD(k)} \otimes \nu^{-1}, \psi)= 1.\]
 \end{thm}
 \vskip 10pt
 
 Applying this theorem to the case at hand, with $\nu = \chi^{-1} \cdot \eta^{-1}$, 
 we immediately deduce [GGP, Conjecture 16.1] (multiplicity one in $L$-packets). 
 In fact, when $\tau$ is not dihedral with respect to $k/k_0$, this theorem also implies [GGP, Conjecture 16.3]. Indeed, in this case, $\tau \boxtimes \chi$ remains irreducible when restricted to $\U(V)$, so that
 \[  \Pi_M = \{ \pi_M, \pi'_M \}. \]
 Moreover, $A_M \cong A_N \cong \Z/2\Z$ and the distinguished character $\chi_0$ of
 $A_M \times A_N$  satisfies
 \[  \chi_0(-1,1) = \chi_0(1,-1) = \epsilon(N(\tau)|_{WD(k)} \otimes \chi \cdot \eta, \psi). \]
 Hence we deduce that
 \[  \text{$\chi_0$ is trivial} \Longleftrightarrow \Hom_{\U(L)}(\pi_M \otimes \eta,\CC) \ne 0 \]
 and 
 \[  \text{$\chi_0$ is nontrivial} \Longleftrightarrow \Hom_{\U(L')}(\pi'_M \otimes \eta,\CC) \ne 0. \] 
 \vskip 10pt
 
 Suppose then that  $\tau$ is dihedral with respect to $k/k_0$, so that 
 \[  N(\tau)|_{WD(k)} =  \alpha + \alpha^{\sigma} \]
 for a character $\alpha$ of $k^{\times}$.  
 In this case, $\tau$ is the sum of two distinct irreducible summands when restricted to $\GL_2(k_0)^+$ and the same holds for its Jacquet-Langlands lift $\tau'$ (if it exists). 
 Now we have
 the following refinement of Theorem \ref{T:saito-tunnell}, due to the third author [P3]. 
  \vskip 10pt
  
  \begin{thm}
  Fix an additive character $\psi$ of $k/k_0$.
  Suppose that $\tau$ is dihedral with respect to $k/k_0$. 
  \vskip 5pt
  
\noindent (i)   The two irreducible summands 
  of $\tau_{\GL_2(k_0)^+}$ can be indexed as  $\tau^{++}$
and $\tau^{--}$ such that the following holds. For any 
character $\beta$ of $k^\times$ with 
\[ \beta|_{k_0^{\times}} = \omega_{\tau} = \alpha|_{k_0^\times} \cdot  \omega_{k/k_0}, \]
we have
\[  \Hom_{k^{\times}}(\tau^{++}, \beta) \ne 0 \Longleftrightarrow 
\epsilon(\alpha \cdot \beta^{-1},\psi) = \epsilon(\alpha^{\sigma} \cdot \beta^{-1},\psi) = +1, \] 
and 
\[  \Hom_{k^{\times}}(\tau^{--}, \beta) \ne 0 \Longleftrightarrow 
\epsilon(\alpha \cdot \beta^{-1},\psi) = \epsilon(\alpha^{\sigma} \cdot \beta^{-1},\psi) = -1. \]

\vskip 10pt

\noindent (ii) Similarly, the two irreducible summands of $\tau'|_{(D^{\times})^+}$ can be labelled  
 as $(\tau')^{+-}$ and $(\tau')^{-+}$ such that the following holds. 
  For any character $\beta$ of $k^\times$ as above,
\[  \Hom_{k^{\times}}((\tau')^{+-}, \beta) \ne 0 \Longleftrightarrow 
\epsilon(\alpha \cdot \beta^{-1},\psi) = +1 \quad \text{and} \quad  \epsilon(\alpha^{\sigma} \cdot \beta^{-1},\psi) = -1,\] 
and 
\[  \Hom_{k^{\times}}((\tau')^{-+}, \beta) \ne 0 \Longleftrightarrow 
\epsilon(\alpha \cdot \beta^{-1},\psi) = -1 \quad \text{and} \quad  \epsilon(\alpha^{\sigma} \cdot \beta^{-1},\psi) = +1. \]
\end{thm} 
\vskip 15pt

This theorem immediately implies [GGP, Conjecture 20.1], which is a variant of [GGP, Conjecture 16.3]. To obtain the full [GGP, Conjecture 16.3], one would need to relate the labelling supplied by the above theorem with the bijection $\Pi_M \leftrightarrow \text{Irr}(A_M)$ determined by the additive character $\psi_2$.  Recall that this bijection was constructed in \S \ref{S:endo-theta} using theta correspondence.  
 \vskip 10pt
 
 To recall the construction, note that when $\tau$ is dihedral with respect to $k/k_0$, we have
 \[  M = M_1 + M_2 \]
 with $M_i$ conjugate-symplectic (not necessarily distinct).
We assume that $M_1 \ne M_2$, since the case $M_1 \cong M_2$ is similar. 
 Then we have
 \[  A_M = \Z/2\Z a_1 \times \Z/2\Z a_2. \] 
 Setting
 \[  \mu = M_1 \quad \text{and} \quad   \chi/\chi^{\sigma} = M_1 \cdot M_2, \]
the packet $\Pi_M$  consists of the representations
 \[  \pi_{\rho}  = \theta_{\psi_{0,2}, \mu, V_B, W_{a}}(\chi) \]
where $W$ is the rank 1 skew-hermitian space of discriminant $\delta$, and $\psi$ and $\psi_0$ are related by
\[  \psi(x) = \psi_0(\frac{1}{2} \cdot {\rm Tr}(\delta \cdot x)). \] 
Moreover,  we have
\[  \rho(a_1) = \epsilon(B) \cdot \omega_{k/k_0}(a) \]
and
\[  \rho(a_2) = \omega_{k/k_0}(a). \]
 
\vskip 10pt
 
Now we consider the see-saw diagram

$$\xymatrix{ \U(L_{-b}+ L_{1})   \ar@{{}-{}}[dd]  & 
\U(W_a)\times \U(W_{a}) 
\ar@{{}-{}}[dd]\\
{} & {}\\
\U(L_{-b}) \times \U(L_1) & \Delta \U(W_a)
}$$
 where $L_{-b}$ denotes the rank 1 hermitian space with discriminant $-b$. 
Note that the rank 2 hermitian space $L_{-b} + L_1$ is isomorphic to $V_B$ with 
\[  \epsilon(B) =  \omega_{k/k_0}(b), \]
and the pair $L_{-b} \subset L_{-b}+ L_1$ is isomorphic to $L_B \subset V_B$
We start with the representation $\chi|_{\U(W_a)}$ on $\Delta \U(W_a)$ and the character $\eta^{-1}$ on $\U(L_{-b})$,  and consider the theta correspondence with respect to the additive character $\psi_{0,2}$. Then 
 the seesaw identity gives

\[  \Hom_{\U(L_{-b})}(\pi_{\rho}, \eta^{-1}) =
\Hom_{\U(W_a)}(\theta_{\psi_0^{-1}, \mu, W_{a}, L_{-b}}(\eta^{-1}) \otimes \omega_{\psi_{0,2}, \mu, W_{a}}, \chi). \] 
Hence, 
\[   \Hom_{\U(L_{-b})}(\pi_{\rho}, \eta^{-1}) \ne 0 \]
if and only if the following two conditions hold:
\vskip 5pt

\begin{enumerate}[(a)]
 \item 
 \[ \theta_{\psi_{0,2}, \mu, W_{a}, L_{-b}}(\eta^{-1})  \ne 0, \]
in which case, $\theta_{\psi_{0,2}, \mu, W_{a}, L_{-b}}(\eta^{-1}) = \eta^{-1}$;
\vskip 5pt

\item 
\[  \Hom_{\U(W_a)}(\eta^{-1} \otimes \omega_{\psi_{0,2}, \mu, W_{a}}, \chi) \ne 0. \] 
\end{enumerate}

\vskip 5pt

\noindent But both (a) and (b) are special cases of Theorem \ref{T:HKS} [HKS, Corollary 8.5]. 
We deduce that (a) holds if and only if
\[  \epsilon(\mu^{-1} \eta^{-1}, \psi_{0,2}({\rm Tr}(\delta -)) = \omega_{k/k_0}(-b) \cdot \omega_{k/k_0}(a) \]
or equivalently
\[  \epsilon(M_1 \otimes N, \psi) = \omega_{k/k_0}(b) \cdot \omega_{k/k_0}(a) = \rho(a_1). \]
Similarly, (b) holds if and only if
\[  \epsilon(\mu^{-1} \cdot \eta \cdot \chi/\chi^{\sigma}, \psi_{0,2}({\rm Tr}(\delta -))) = \omega_{k/k_0}(a), \]
  or equivalently
  \[  \epsilon(M_2 \otimes N,\psi)  = \omega_{k/k_0}(a) = \rho(a_2). \]
 \vskip 10pt
 
\noindent  Thus, we conclude that 
\[   \Hom_{\U(L_b)}(\pi_{\rho}, \eta^{-1}) \ne 0 \]
if and only if $\rho$ is the distinguished character $\chi_0$ of [GGP, Conjecture 16.3], where the local root numbers are computed using the additive character $\psi$ of $k/k_0$.
  
 \vskip 15pt
 \section{Theta correspondence for $\U(2) \times \U(2)$}  \label{S:theta-U(2)xU(2)}
 
 Before moving on to the next case of [GGP, Conjecture 16.3], we need to establish some results about the theta correspondence for $\U(2) \times \U(2)$. More precisely, let $V_B$ be the rank 2 hermitian space introduced in \S \ref{S:LV}. and let $W_{B'}$ be the rank 2 skew-hermitian space obtained from $V_{B'}$ be scaling by the trace zero element $\delta \in k^{\times}$.
 Then we are interested in establishing the theta correspondence for the dual pair
 \[  \U(V_B) \times \U(W_{B'}) \]
 relative to the data $(\psi_0, \mu)$.
 
 \vskip 10pt

 The first result is the following proposition due to Harris [Ha, Lemma 4.3.3] and Konno-Konno [KK, Prop. 5.3 and Thm. 5.4].
 \vskip 5pt
 
 \begin{prop}  \label{P:harris}
 Let $M$ be a 2-dimensional conjugate-symplectic representation of $WD(k)$ which gives rise to a $L$-packet $\Pi_{M,B}$ for $\U(V_B)$ and $\Pi_{M,B'}$ for $\U(W_{B'})$. 
 \vskip 5pt
 
 \noindent (i) For any $\pi \in \Pi_{M,B}$, 
 \[  \theta_{\psi_0, V_B, W_{B'}, \mu}(\pi) \ne 0 \Longleftrightarrow \epsilon(M \otimes \mu^{-2}, \psi) =  \epsilon(B) \cdot \epsilon(B'). \]
 Note that the root number above is independent of the choice of the additive character $\psi$ of $k/k_0$.

 \noindent (ii) If the condition of (i) holds, then $\theta_{\psi_0, V_B, W_{B'}, \mu}(\pi)$ belongs to $\Pi_{M, B'}$. In other words, the theta correspondence is the identity map on $L$-parameters.
 
 \vskip 10pt
 
 \noindent Thus, under the theta correspondence  for $(\psi_0,\mu)$, there is a unique $B'$ such that the theta lift gives a bijection
 \[  \Pi_{M,B} \longleftrightarrow \Pi_{M, B'}. \]
\end{prop}
 
 \vskip 5pt
 
 If the parameter $M$ is non-dihedral (with respect to $k/k_0$), then $\# \Pi_{M,B} = 0$ or $1$. Hence the above proposition completely determines the theta lift of the representations in $\Pi_M$. When $M$ is dihedral with respect to 
$k/k_0$, then
 $\# \Pi_{M,B} = 0$ or $2$, and in the latter case, there are two possible bijections
 \[  \Pi_{M,B} \longleftrightarrow \Pi_{M, B'}, \] 
 which the above proposition does not resolve. In [P5], the third author has given a precise conjecture addressing this issue. The following theorem confirms the conjecture in [P5] for this case:
  \vskip 10pt
  
 \begin{thm}  \label{T:theta-U2}
 Suppose that $M = M_1 + M_2$ is dihedral with respect to $k/k_0$.
 Fix the additive character $\psi$ of $k/k_0$ which gives  bijections
 \[  \Pi_M \longleftrightarrow \text{Irr}(A_M), \]
 and let $\psi_0$ be the additive character of $k_0$ such that
 \[  \psi(x) = \psi_0(\frac{1}{2} \cdot {\rm Tr}(\delta x)). \]
 Then the permutation of $\Pi_M$ induced by the theta correspondence associated to $(\psi_0, \mu, \delta)$ is given by multiplication by the character $\rho_0$ of $A_M$ defined by
 \[  \rho_0(a_i) =  \epsilon(M_i \otimes \mu^{-2}, \psi_{2}) \] 
 with
  \[ \psi_{2}(x) = \psi_0({\rm Tr}(\delta  x)). \]
 \end{thm}

\vskip 5pt

\begin{proof}
 Consider first the case where $B'$ is split whereas $B$ is arbitrary. In this case, then two elements in $\Pi_{M,B'}$ can be distinguished by the Whittaker models they support. Computing Whittaker models of the Weil representation $\omega_{\psi_0, V_B, W_{B'},\mu}$, one sees
 that for $\pi_{\rho}  \in \Pi_{M,B}$, 
 the representation  $\theta_{\psi_0, V_B, W_{B'}, \mu}(\pi_{\rho})$  of $\U(W_{B'})$ is $\psi_0$-generic if and only if
 \[  \Hom_{\U(L_B)}(\pi_{\rho}^{\vee} , \mu^{-2}) \ne 0. \] 
By the result of the previous section, this holds if and only if
 \[  \rho(a_1) = \epsilon(M_1 \otimes \mu^{-2}, \psi_{2})  \quad \text{and} \quad \rho(a_2) =  
 \epsilon(M_2 \otimes \mu^{-2}, \psi_{2}), \]
as desired. This establishes the result when one of $B$ or $B'$ is split.

 \vskip 5pt
 
 The only remaining case is where $B$ and $B'$ are both non-split, so that  
 \[  \epsilon(M_1\otimes \mu^{-2}, \psi)\cdot  \epsilon(M_2 \otimes \mu^{-2}, \psi)   = 1. \]
 In this case, the desired result can be proved by a global method. We give a brief sketch of this.
 \vskip 5pt
 
Let $\pi$ be a representation in $\Pi_{M,B}$, so that $\theta_{\psi_0,\mu}(\pi)$ belongs to $\Pi_{M,B}$ also.  We may find:
\vskip 5pt

\begin{enumerate}
\item a number field $F$ with at least one real place $v_{\infty}$ and such that $F_{v_0} = k_0$ for some finite place $v_0$ of $F$;
\vskip 5pt

\item an additive character $\Psi$ of $\mathbb{A}_F/F$ such that $\Psi_{v_0} = \psi_0$;
\vskip 5pt

\item a quadratic extension $E$ of $F$ such that $E_{v_{\infty}} \cong \CC$ and $E_{v_0} \cong k$;
\vskip 5pt
\item a trace zero element $\Delta \in E$ such that $\Delta_v = \delta$;
\vskip 5pt

\item an idele class character $\Sigma$ of $E^{\times}$ such that $\Sigma_{v_0} = \mu$ and $\Sigma|_{\mathbb{A}_F^{\times}} = \omega_{E/F}$;
\vskip 5pt

\item a global quaternion algebra $\mathbb{B}$ over $F$ ramified precisely at $\{ v_{\infty}, v_0 \}$, so that  $\mathbb{B}_{v_0} = B$; this gives a hermitian space $V_{\mathbb{B}}$ over $F$ which is isomorphic to $V_B$ over $F_{v_0}$;
\vskip 5pt

\item a cuspidal representation $\Pi$ of $\U(V_{\mathbb{B}})$ such that 
\vskip 5pt

(a) $\Pi_{v_0} = \pi$;
\vskip 5pt

(b) $\Pi$ belongs to a global endoscopic packet;
\vskip 5pt

(c) $L(BC_{E/F}(\Pi) \otimes \Sigma^{-2}, 1/2) \ne 0$ 
\end{enumerate} 
 \vskip 5pt
 
 In particular, it follows that the set
 \[  S = \{ v: \epsilon( BC_{E_v/F_v}(\Pi_v) \otimes \Sigma_v, \psi)  = -1 \} \]
 has even cardinality and does not contain the place $v_0$. 
 Let $\mathbb{B}'$ be the quaternion algebra over $F$ such that 
 \[ \epsilon( \mathbb{B}'_v) \ne \epsilon(\mathbb{B}_v) \Longleftrightarrow v \in S. \]
 In other words, $\mathbb{B}'$ is obtained from $\mathbb{B}$ by switching the local 
 invariants of $\mathbb{B}$ at the set $S$. Since $v_0 \notin S$, we have
 \[  \mathbb{B}'_{v_0} \cong  B. \]
 Moreover, for each place $v$ of $F$, 
 \[  \Theta_{\Psi_v, \Sigma_v, V_{\mathbb{B}_v}, W_{\mathbb{B}'_v}}(\Pi_v) \ne 0. \] 
 By [Ha], The non-vanishing of the central $L$-value above implies that the global theta lift is non-vanishing as well:
 \[  \Theta_{\Psi,\Sigma,   V_{\mathbb{B}}, W_{\mathbb{B}'}}(\Pi) \ne 0.\]
  \vskip 10pt
  
  Now the assertion of the theorem has been checked for all finite places of $F$ outside $v_0$, since at least one of $\mathbb{B}_v$ or $\mathbb{B}'_v$ is split at any $v \ne v_0$. At the archimedean places,  the result of the theorem is also known (c.f. [Pa] for example). If the result of the theorem is not true at the place $v_0$, we would have a cuspidal representation $\Theta_{\Psi,\Sigma,   V_{\mathbb{B}}, W_{\mathbb{B}'}}(\Pi)$ of $\U(W_{\mathbb{B}'})$ which violates the Langlands-Arthur multiplicity formula for global endoscopic packets of $\U(2)$. This gives the desired contradiction.
  \vskip 10pt
  
  For example, suppose that $S$ is empty so that $\mathbb{B} = \mathbb{B}'$. Then if the result of the theorem holds at all $v \ne v_0$ but fails at $v_0$, the cuspidal representation
   $\Theta_{\Psi,\Sigma,   V_{\mathbb{B}}, W_{\mathbb{B}'}}(\Pi)$ of $\U(W_{\mathbb{B}'})$ would differ from the cuspidal representation $\Pi$ at an odd number of places $v$. This is a contradiction.
  
 \end{proof}
 
 \vskip 15pt

\section{Trilinear forms for $\U(2)$}  \label{S:trilinear}

In this section, we return to the skew-hermitian case of [GGP, Conjecture 16.3]. In particular,
we consider the case when
\[  W \subseteq V \quad \text{with $\dim W = \dim V = 2$.} \]
\vskip 10pt

Thus, let $W_{B, \delta}$ be the rank $2$ skew-hermitian case obtained from $V_B$ by scaling by $\delta$.
Fix an additive character $\psi_0$ of $k_0$, and a character $\mu$ of $k^{\times}$ so that 
\[  \mu|_{k_0^{\times}} = \omega_{k/k_0}. \]
This determines Weil representations $\omega_{\psi_0, \mu}$ for $\U(W_B)$.  
Given two conjugate-symplectic representations $M$ and $N$ of dimension $2$, with corresponding Vogan packet $\Pi_M$ and $\Pi_N$, we are interested in computing
\[ \Hom_{\U(W_B)}(\pi_M \otimes \pi_N \otimes \overline{\omega_{\psi, \mu}}, \CC) \]
as $\pi_M$ and $\pi_N$ vary over all representations in $\Pi_M$ and $\Pi_N$. 
\vskip 10pt

Note that the representation $\omega_{\psi_0,\mu}$ is not an irreducible representation of $\U(W_B)$. However, we may decompose $\omega_{\psi_0 ,\mu}$ according to central characters
\[  \omega_{\psi_0,\mu} = \bigoplus_{\chi} \omega_{\psi_0, \mu}[\chi] \]
 as $\chi$ runs over characters of $Z_{\U(W_B)} \cong k^{\times}/k_0^{\times}$. In fact, this decomposition is simply the decomposition of the Weil representation for the dual pair 
 $\U(V_1)  \times \U(W_B)$. Thus, each summand $\omega_{\psi_0,\mu}[\chi]$ is an irreducible
 representation of $\U(W_B)$. Moreover, it belongs to an endoscopic packet of $\U(W_B)$ constructed in Proposition \ref{P:endo-U2}.
 \vskip 10pt
 
 Now, because of central character reasons, it is clear that
 \[   \Hom_{\U(W_B)}(\pi_M \otimes \pi_N \otimes \overline{\omega_{\psi_0, \mu}[\chi]}, \CC) = 0 \]
 unless
 \[   \det M \cdot \det N  = \chi. \]
 For this $\chi$, we have
 \[   \Hom_{\U(W_B)}(\pi_M \otimes \pi_N \otimes \overline{\omega_{\psi_0, \mu}}, \CC) 
 =  \Hom_{\U(W_B)}(\pi_M \otimes \pi_N \otimes \overline{\omega_{\psi_0, \mu}[\chi]}, \CC). \]
 In particular, [GGP, Conjecture 16.3] amounts to a question about invariant trilinear forms on $\U(W_B)$.
  \vskip 10pt

 Given that the group $\U(W_B)$ can be described in terms of $\GL_2(k_0)$ and its inner form, 
 we shall see this question can be related to a question about invariant trilinear forms for $\GL_2$. This has been addressed in a series of paper by the third author [P1,2,6]; we recall his result here:
 \vskip 10pt
 
 \begin{thm} \label{T:trilinear}
 Let $N_1$, $N_2$ and $N_3$ be 2-dimensional representations of $WD(k_0)$, with associated representations  $\pi_{i,B}$ of $B^{\times}$.
 Assume that $\det N_1 \cdot \det N_2 \cdot \det N_3  =1$. Then
 \[   \sum_B   \dim \Hom_{B^{\times}}(\pi_{1,B} \otimes \pi_{2,B} \otimes \pi_{3,B}, \CC) =1. \]
  Moreover, 
  \[   \Hom_{B^{\times}}(\pi_{1,B} \otimes \pi_{2,B} \otimes \pi_{3,B}, \CC) \ne 0 \Longleftrightarrow 
  \epsilon(N_1 \otimes N_2 \otimes N_3, \psi) = \epsilon(B). \]
    \end{thm}
  \vskip 10pt
  
  To apply this theorem to the case of $\U(W_B)$,  we need to consider the group $(B^{\times})^+$ and 
  calculate 
$${\rm dim}{\rm Hom}_{(B^{\times})^+}(\pi_1 \otimes \pi_2 \otimes \pi_3,  \CC) .
$$ 
\vskip 10pt

More generally,  let $G$ be a subgroup of $\GL_2(k_0)$ containing $\SL_2(k_0)$. The group 
$G$ is uniquely determined by the subgroup 
\[  k^{\times}_G \subset k_0^\times \]
consisting of determinants of elements of $G$. Thus, for any quaternion algebra $B$, it makes sense to define a 
corresponding subgroup $G_B$ inside $B^\times$ containing $\SL_1(B)$.
Restricting representations of $B^{\times}$ to $G_B$, 
one gets a notion of $L$-packet of representations of $G_B$.
It is known that representations of $\GL_2(k_0)$ restrict to $G$ with multiplicity 1, but
this need not be the case for representations of $B^\times$ if $B$ is non-split. 
For a representation $\pi_B$ of $G_B$, let $m(\pi_B)$ denote the multiplicity
with which it appears in the restriction of an irreducible  representation of $B^\times$.   
\vskip 10pt

Now we have:
\begin{thm}  \label{T:trilinear2}
For $i =1$, $2$ and $3$, let $N_i$ be a 2-dimensional representation of $WD(k_0)$ with associated representation $\tilde{\pi}_{B,i}$ of $B^{\times}$. Assume that $\prod_i \det N_i  =1$.
Then
\[  \sum_B  \dim \Hom_{G_B}(\tilde{\pi}_{B,1} \otimes \tilde{\pi}_{B,2} \otimes \tilde{\pi}_{B,3}, \CC) 
= \# (k_0^\times/k_0^{\times 2}k^\times_{G}). \]
In particular, 
\[
\sum_B \sum_{\pi_{B,1}, \pi_{B,2}, \pi_{B,3}} m(\pi_{B,1}) \cdot m(\pi_{B,2}) \cdot m(\pi_{B,3}) \cdot \dim \Hom_{G_B}(\pi_{B,1} \otimes \pi_{B,2} \otimes \pi_{B,3}, \CC) \]
is equal to
\[  \# (k_0^\times/k_0^{\times 2}k^\times_{G}), \]
where the inner sum is taken over irreducible representations $\pi_{B,i}$ 
of $G_B$ which are contained in
the representations $\tilde{\pi}_i$ of $B^{\times}$.
\end{thm}

\begin{proof}
Clearly,
$${\rm Hom}_{G_B}(\tilde{\pi}_{B,1} \otimes \tilde{\pi}_{B,2} \otimes \tilde{\pi}_{B,3}, \CC) 
\cong \sum_{\chi:k_0^\times/k^\times_G \rightarrow {\Bbb Z}/2} {\rm Hom}_{G_B}(\tilde{\pi}_{B,1} \otimes \tilde{\pi}_{B,2} \otimes \tilde{\pi}_{B,3}, \CC_\chi), $$
where the $\chi$'s range over characters of $B^{\times}$ trivial on $G_B$, 
and $\CC_{\chi}$ 
denotes the 1-dimensional representation $\chi \circ \mathbb{N}_B$ of $B^{\times}$.  
But by Theorem \ref{T:trilinear}, we have
\[
\sum_B \dim \Hom_{G_B}(\tilde{\pi}_{B,1} \otimes \tilde{\pi}_{B,2} \otimes \tilde{\pi}_{B,3}, 
\CC_\chi)  = 1, \]
for all characters $\chi$ of order $\leq 2$ (by absorbing $\chi$ in one of the 
$\pi_i$'s). Adding up the contribution of the  various $\chi$'s, we get the conclusion
of the theorem.

\end{proof}
\vskip 10pt

Specializing this theorem to the case $G_B = (B^{\times})^+$ and noting that,  in this case, $m(\pi_{B,i}) = 1$ for each $B$, we obtain:
\vskip 5pt

\begin{cor} 
In the context of Theorem \ref{T:trilinear2},  let $G= \GL_2(k_0)^+$. Then one has,
\[   
\sum_B \sum_{\pi_{B,1}, \pi_{B,2}, \pi_{B,3}} \dim \Hom_{G_B}(\pi_{B,1} \otimes \pi_{B,2} \otimes \pi_{B,3}, \CC)
  = 2, \]
 where the inner sum is taken over irreducible representations $\pi_{B,i}$
of $G_B$  which are contained in
the representations $\tilde{\pi}_{B,i}$ of $B^{\times}$.

\end{cor}

\vskip 10pt

We can now apply the corollary to the group $\GU^+(W_B)$ or equivalently $\U(W_B)$. 
\vskip 5pt

 \begin{cor}
 Let $M_i$ be conjugate-symplectic representations of $WD(k)$ with associated $L$-packet $\Pi_{M_i,B}$ of $\U(W_B)$.  Assume that $\det M_1 \cdot \det M_2 \cdot \det M_3 = 1$. Then 
 \vskip 5pt
 
 \noindent (i) 
 \[  \sum_B \sum_{\pi_i \in \Pi_{M_i,B}}
\dim \Hom_{\U(W_B)}(\pi_1 \otimes \pi_2 \otimes \pi_3, \CC) 
  = 2. \]
\vskip 5pt

\noindent (ii) If one of the $M_i$'s, say $M_1$,  is dihedral with respect 
to $k/k_0$, so that 
$\# \Pi_{M_1,B_0} =2$ for $B_0$ split, then  
$$\dim \Hom_{\U(W_B)}({\pi}_1 \otimes {\pi}_2 \otimes {\pi}_3,  \CC)  \leq 1$$
for each $B$. If the above Hom space is nonzero, then
$$\dim \Hom_{\U(W_{B'})}({\pi}'_1 \otimes {\pi}'_2 \otimes {\pi}'_3,  \CC)   =0 $$ 
for $B' \ne B$. 
 \end{cor}

\begin{proof} 
The first assertion follows immediately from the previous corollary and the definition of $L$-packets for $\U(W)$ given in \S \ref{S:LV}. To deduce the last assertion, note that
if 
\[ \Hom_{\U(W_B)}(\pi_1 \otimes {\pi}_2 \otimes {\pi}_3, \CC) \not = 0, \]
then we also have
\[   \Hom_{\U(W_B)}(\pi_1^c \otimes \pi^c_2 \otimes \pi^c_3, \CC) \ne 0, \]
where $\pi^c_i$  denotes the 
conjugate of $\pi_i$ by an element $c \in \GU(W_B) \smallsetminus \GU^+(W_B)$.   
Since 
\[ \dim \Hom_{\U(W_B)}(\pi_1 \otimes {\pi}_2 \otimes {\pi}_3, \CC)  + \dim \Hom_{\U(W_B)}(\pi^c_1 \otimes \pi_2^c \otimes \pi_3^c, \CC) \leq 2, \]
 each of these dimensions must be equal to 1, and all other Hom spaces must be 0.
\end{proof}

\vskip 10pt

 \noindent{\bf Remark:} 
Since $k_0^\times/k_0^{\times 2}$ is a 2-group
whose cardinality can be made arbitrarily large by choosing $k_0$ 
appropriately, and since the $L$-packet of representations of $\SL_2(k)$
is bounded by 4 [LL], it follows that 
\[ 
{\rm dim}{\rm Hom}_{\SL_2(k)}({\pi}_1 \otimes {\pi}_2 \otimes {\pi}_3, {\Bbb C}) \]
 can be made arbitrarily large.
\vskip 15pt

Now we can return to [GGP, Conjecture 16.3], so that 
$M$ and $N$ are two 2-dimensional conjugate-symplectic representations of $WD(k)$ which determine Vogan packets $\Pi_M$ and $\Pi_N$ of $\U(W_{B})$. For a fixed additive character $\psi_0$ of $k_0$, we have obtained a bijection
\[  J(\psi_0): \Pi_M \longleftrightarrow \text{Irr}(A_M) \]
and similarly for $\Pi_N$. We are interested in computing 
\[  \Hom_{\U(W_B)}(\pi_M \otimes \pi_N \otimes \overline{\omega_{\psi_0,\mu}}) \]
for $\pi_M\in \Pi_M$ and $\pi_N \in \Pi_N$. 
\vskip 10pt

If $M$ and $N$ are non-dihedral (with respect to $k/k_0$), so that $\Pi_M$ and $\Pi_N$ both contain at most one representation of each $\U(W_B)$ (as $B$ varies), then [GGP, Conjecture 16.3] is a consequence of Theorem \ref{T:trilinear}. Indeed, we have
\[  A_M \times A_N = \Z/2\Z \times \Z/2\Z \]
and the distinguished character $\chi_0$ satisfies
\[  \chi_0(-1,1) = \chi_0(1,-1) = \epsilon(M \otimes N(\mu^{-1}),\psi) \]
for any character $\psi$ of $k/k_0$. 
On the other hand, if $\Pi_M$ is obtained by the restriction of the representation $\tau_M \boxtimes \chi_M$ of $\GU(W_B)$ and $\Pi_N$ is obtained from $\tau_N \boxtimes \chi_N$, then
the epsilon factor occurring in Theorem \ref{T:trilinear} is
\begin{align}
  &\epsilon( \rho_M \otimes  \rho_N \otimes \text{Ind} (\mu^{-1} \chi_M \chi_N), \psi_0) \notag \\
=& \epsilon ( \rho_M|_{WD(k)} \otimes \rho_N|_{WD(k)} \otimes \mu^{-1} \cdot \chi_M \cdot \chi_N, \psi_0({\rm Tr})) \notag \\
= &\epsilon(M \otimes N(\mu^{-1}), \psi_0({\rm Tr})) \notag \\
=& \epsilon(M \otimes N(\mu^{-1}),\psi). \notag
\end{align}
This verifies [GGP, Conjecture 16.3] in this case.
\vskip 10pt

When at least one of $M$ or $N$ is dihedral with respect to $k/k_0$, we may appeal to the  theta correspondence. 
Since the case when exactly one of them is dihedral with respect to $k/k_0$ is similar and easier, we shall give the details only when both $M$ and $N$ are dihedral with respect to $k/k_0$ Thus, let
\[  M = M_1 + M_2 \quad \text{and} \quad  N = N_1 + N_2, \]
with $M_i$ and $N_i$ conjugate-symplectic (not necessarily distinct), and write their component groups as
\[  A_M = \Z/2\Z e_1 \times \Z/2\Z e_2 \quad \text{and} \quad A_N = \Z/2\Z f_1 \times \Z/2\Z f_2. \]
In this case, the packet $\Pi_M$ can be obtained by theta correspondence from $\U(1)$. 
Set 
\[  \nu = M_1 \]
and 
\[  M_1 \cdot M_2 = \eta/\eta^{\sigma}, \]
for some character $\eta$ of $k^{\times}$.
If $L_a$ denote the rank $1$ hermitian space with discriminant $a$, then
\[  \Pi_M = \{  \theta_{\psi_0,\nu, W_{B,\delta}, L_a}(\eta|_{\U(L_a)}):  \quad a \in k_0^{\times}/\mathbb{N}k^{\times}, \, \epsilon(B) = \pm 1 \}. \] 
Relative to the additive character $\psi$ of $k/k_0$, we have the labelling
\[   \pi_{\rho_M} = \theta_{\psi_0,\nu, W_{B,\delta}, L_a}(\eta|_{\U(L_a)}) \]
if and only if
\[  \rho_M(e_1) = \epsilon(B) \cdot \omega_{k/k_0}(a) \quad \text{and} \quad \rho_M(e_2) = \omega_{k/k_0}(a). \]
Similarly, a representation in $\Pi_N$ has the form $\pi_{\rho_N}$, so that
\[  \rho_N(f_1) \cdot \rho_N(f_2) = \epsilon(B). \]
\vskip 10pt

Now consider the see-saw diagram

$$ \xymatrix{ \U(L_a+L_{-1})  \ar@{-}[dd]  & \U(W_B)\times \U(W_B) \ar@{-}[dd]\\
 & \\
 \U(L_a) \times \U(L_{-1})    & \Delta \U(W_B)}
 $$
and note that the rank 2 hermitian space $L_a + L_{-1}$ is isomorphic to $V_{B'}$
with $\epsilon(B') = \omega_{k/k_0}(a)$.
We start with the representation $\eta|_{\U(L_a)}$ of $\U(L_a)$, so that the representation we obtain on $\U(W_B)$ is precisely
\[  \pi_{\rho_M} = \theta_{\psi_0,\nu, W_{B,\delta}, L_a}(\eta|_{\U(L_a)}). \]
On the other side of the see-saw, we start with 
 the representation $\mu \cdot \nu \cdot \pi_{\rho_N}^{\vee}$ of $\U(W_B)$.
 Note that taking contragredient has the following effect on the Vogan parameterization: 
  for any character $\rho_N$ of $A_N$, the representation $\pi_{\rho_N}$ has Vogan parameter
  \[  (N^{\vee}, \rho_N \cdot \beta_0 ) \]
  where $\beta_0$ is the character of $A_{N^{\vee}} = A_N$ given by
  \[  \beta_0(b_i) = \omega_{k/k_0}(-1). \]
 \vskip 15pt

Now the see-saw identity gives:

\[
\Hom_{\U(W_B)}(\pi_{\rho_M} \otimes 
\omega_{\psi_0^{-1},\nu, W_B}, \mu \cdot \nu \cdot  \pi_{\rho_N}^{\vee})  \]
\[  = \Hom_{\U(L_a)}(\Theta_{\psi_0, \nu^2,  W_B, L_a+L_{-1}}(\mu \nu \pi_{\rho_N}^{\vee}), \eta|_{\U(L_a)}). \]
Since
\[  \omega_{\psi_0^{-1}, \nu, W_B} = \omega_{\psi_0^{-1}, \mu^{-1}, W_B} \otimes \mu \nu= \overline{\omega_{\psi_0,\mu}} \otimes \mu \nu,  \]
we see that the LHS of this identity is equal to the desired space
\[  \Hom_{\U(W_B)} (\pi_{\rho_M} \otimes \pi_{\rho_N} \otimes \overline{\omega_{\psi_0,\mu}},  \CC). \]
On the other hand, the RHS is nonzero if and only if conditions (a) and (b) below are satisfied: 
\vskip 5pt

\begin{enumerate}[(a)]
\item $\Theta_{\psi_0, \nu^2,  W_B, L_a+L_{-1}}(\mu \nu (\pi_{\rho_N})^{\vee} ) \ne 0$. According to Theorem \ref{T:theta-U2},  this holds 
 if and only if
\[  \epsilon(N^{\vee} \otimes \mu \nu \nu^{-2},\psi)  = \epsilon(B) \cdot  \omega_{k/k_0}(a), \]
or equivalently
\[  \epsilon(N \otimes M_1(\mu^{-1}), \psi) = \epsilon(B) \cdot \omega_{k/k_0}(a) = \rho_M(e_1). \]
If this is satisfied, then by Theorem \ref{T:theta-U2}, the  theta lift is equal to the representation
\[  \pi_{\rho_{N^{\vee}}} \cdot \mu \nu \]
of $\U(L_a+L_{-1})$,  with
 \[  \rho_{N^{\vee}}(f_i) = \rho_N(f_i) \cdot \omega_{k/k_0}(-1) \cdot  \epsilon(N_i \otimes M_1(\mu^{-1}), \psi_{-2}). \]
 \vskip 10pt

\item $\Hom_{\U(L_a)}(\pi_{\rho_{N^{\vee}}} \mu \nu , \eta/\eta^{\sigma}) \ne 0$. This is a branching problem for $\U(2) \times \U(1)$
 which we have resolved in \S \ref{S:U(2)xU(1)}.  Using the results there, we see that the desired non-vanishing holds  if and only if
\[ \rho_{N^{\vee}}(f_i) =  \rho_N(f_i)  \cdot \omega_{k/k_0}(-1) \cdot \epsilon(N_i \otimes M_1(\mu^{-1}), \psi_{-2}) = \epsilon(N_i \otimes M_2(\mu^{-1}), \psi_{2}) \]
or equivalently
 \[  \rho_N(f_i)  = \epsilon(N_1 \otimes M(\mu^{-1}), \psi_2) =  \epsilon(N_1 \otimes M(\mu^{-1}), \psi). 
 \]
 \end{enumerate}
 \vskip 10pt
 
 Finally, since 
 \[  \rho_M(-1) = \rho_N(-1) = \epsilon(B), \]
 we conclude that
 \[ \rho_M(e_2) =  \epsilon(N \otimes M_2(\mu^{-1}), \psi). \]
 Thus we conclude that 
 \[   \Hom_{\U(W_B)} (\pi_{\rho_M} \otimes \pi_{\rho_N} \otimes \overline{\omega_{\psi_0,\mu}}, \CC) \ne 0 \] 
if and only if $\rho_M \times \rho_N$ is the distinguished character $\chi_0$ of [GGP, Conjecture 16.3].

\vskip 15pt
\section{Restriction from $\U(3)$ to $\U(2)$: endoscopic case}
In this section, we consider the restriction problem for $\U(3) \times \U(2)$.
Using theta correspondence, we establish [GGP, Conjecture 16.3]  for endoscopic packets of 
$\U(3)$. In the following section, we shall consider the stable packets of $\U(3)$.
\vskip 10pt

We fix a pair
\[  W \subset V \]
of split hermitian spaces of dimension $2$ and $3$ respectively. Without loss of generality, we assume that $\disc V = 1$, so that $V/W$ is a rank $1$ hermitian space $L_1$  of discriminant $1$.  
Let $W' \subset V'$ be hermitian spaces of dimension $2$ and $3$, 
such that $V/W \cong V'/W'$. 
\vskip 10pt
 
More concretely, for each quaternion algebra $B$ over $k_0$, 
we have a rank 2 hermitian space $V_B$. Then the rank 3 hermitian space
\[  V_{B,b} = V_B +  L_b \]
has discriminant satisfying
\[  \omega_{k/k_0}(\disc (V_{B,b}))  = \epsilon(B) \cdot \omega_{k/k_0}(b). \] 
If we take $b = 1$, then as $B$ varies,  the pair
\[  V_B \subset V_{B,1} \]
gives the pairs $W \subset V$ and $W' \subset V'$.
\vskip 5pt
 
 Suppose first that $N$ is a 2-dimensional conjugate-symplectic representation of $WD(k)$ with associated Vogan packet $\Pi_{N}$ of $\U(V_B)$. If $N = \oplus_i N_i$,  then we write
 \[  A_N = \prod_i A_{N_i} = \prod_i \Z/2\Z f_i.\]
 For the fixed given additive character $\psi$ of $k/k_0$, we translate $\psi$ by $-2 \cdot \disc(V) = -2$ and use the resulting character $\psi_{-2}$ to fix the Vogan parameterization
 \[  J(\psi_{-2}): \Pi_N \longleftrightarrow \text{Irr}(A_N). \]
 \vskip 10pt

Now consider a 3-dimensional conjugate-orthogonal representation 
\[  M = M_1 + M_2 \]
with $\dim M_i = i$ and such that each $M_i$ is conjugate-orthogonal. Unless, $M \cong 3 M_1$, we may further assume that $M_1$ does not occur in $M_2$. We shall assume that this is the case, since the other case is similarly handled. Then
\[  A_M = A_{M_1} \times A_{M_2} \]
and we write:
\[  A_{M_1} = \Z/2\Z e \quad \text{and} \quad A_{M_2} = \prod_i \Z/2\Z e_i \]
if $M_2 = \oplus_i M_{2,i}$.

\vskip 5pt

Moreover, we shall assume that the conjugate-orthogonal character $M_1$ has a conjugate-symplectic square root. This can be achieved by twisting $M$, and since this twist can be absorbed into $N$ for the purpose of the restriction problem, there is no loss of generality in making this assumption on $M_1$. Under this assumption on $M_1$,
we have described in \S \ref{S:endo-theta} a construction of the Vogan packet
$\Pi_M$ as well as a bijection
\[  \Pi_M \longleftrightarrow  \text{Irr}(A_M) \]
which is canonical in this case (i.e. independent of the choice of any additive character).
To recall the construction briefly, we set
\[  M_1 = \mu^2 \]
for some conjugate-symplectic character $\mu$ and set
\[  N' = M_2 \cdot \mu, \]
so that $N'$ is conjugate-symplectic and $A_{N'} = A_{M_2}$. Then, for quaternion algebras $B$ and $B'$ over $k_0$, one considers the theta correspondence for 
\[  \U(W_{B'})  \times \U(V_{B,1}) \]
relative to the data $(\psi_{0,-2}, \mu)$, where $\psi_0$ is our fixed additive character of $k_0$.
The packet $\Pi_M$ is then the theta lift of the packet $\Pi_{N'}$ of $\U(W_{B'})$. For the labelling of the representations in $\Pi_M$ by $\text{Irr}(A_M)$, we refer the reader to the end of \S \ref{S:endo-theta}. 

\vskip 10pt

 Now we would like to determine
\[  \Hom_{\U(V_B)}(\pi_M \otimes \pi_N, \CC), \]
for $\pi_M \in \Pi_M$ and $\pi_N \in \Pi_N$.
We examine this restriction problem using the see-saw diagram


$$\xymatrix{ \U(V_B+L_1)   \ar@{-}[dd]  & \U(W_{B'}) \times \U(W_{B'}) \ar@{-}[dd]\\
 & \\
 \U(V_B) \times \U(L_1)    & \U(W_{B'}).
}$$

On $\U(W_{B'})$, we start with a representation $\pi_{N'}^{\eta} \in \Pi_{N'}$ indexed by a character $\eta$ of $A_{N'}$, so that
\[  \eta(-1) = \epsilon(B'). \]
On $\U(V_B)$, we start with a representation $\pi_{\rho_N}^{\vee}$ associated to a character $\rho_N$ of $A_N$, so that
\[  \rho_N(-1) = \epsilon(B). \]
Then we have the see-saw identity:

\[  \Hom_{\U(V_B)}(\Theta_{\psi_{0,-2},\mu}(\pi_{N'}^{\eta}) \otimes \pi_{\rho_N}, \CC)
  = \Hom_{\U(W_{B'})}(\Theta_{\psi_{0,-2}, \mu^2, V_B, W_{B'}}(\pi_{\rho_N}^{\vee}) \otimes \omega_{\psi_{0,-2}, \mu, L_1, W_{B'}}, \pi_{N'}^{\eta}). \]
\vskip 5pt

\noindent Note that 
\[  \pi_{\rho_M} =  \theta_{\psi_{0,-2},\mu}(\pi_{N'}^{\eta}) \]
 with 
 \[  \rho_M|_{A_{N'}} = \eta \quad \text{and} \quad \rho_M(e) = \epsilon(B') \cdot \eta(-1) = \epsilon(B) \cdot \epsilon(B'). \]  
Moreover, $\pi_{\rho_N}^{\vee}$ has Vogan parameter  (relative to $J(\psi_{0,-2})$)
\[  (N^{\vee}, \rho_{N^{\vee}}) = (N^{\vee}, \rho_N \cdot \beta_0) \]
 with 
\[  \beta_0(f_i) = \omega_{k/k_0}(-1). \]
Then the see-saw identity reads:

\[  \Hom_{\U(V_B)}(\pi_{\rho_M} \otimes \pi_{\rho_N}, \CC)
  = \Hom_{\U(W_{B'})}(\theta_{\psi_{0,-2}, \mu^2, V_B, W_{B'}}(\pi_{\rho_{N^{\vee}}}) \otimes \omega_{\psi_{0,-2}, \mu, W_{B'}}, \pi_{N'}^{\eta}). \]  
\vskip 5pt

The RHS is nonzero if and only if (i) and (ii) below hold.
\vskip 5pt
\begin{enumerate}[(i)]
\item $\theta_{\psi_{0,-2}, \mu^2, V_B, W_{B'}}(\pi_{\rho_{N^{\vee}}}) \ne 0$. By proposition \ref{P:harris}, this holds if and only if
\[  \epsilon(N^{\vee} \mu^{-2}, \psi_{-2}) = \epsilon(B) \cdot \epsilon(B') = \rho_M(e), \]
or equivalently
\[  \epsilon(N \otimes M_1, \psi) = \rho_M(e). \]
Moreover, by Theorem \ref{T:theta-U2}, when this holds, we have
\[ \theta_{\psi_0, \mu^2, V_B, W_{B'}}((\pi_{\rho_{N^{\vee}}})
= \pi_{\rho_{N^{\vee}} \cdot \rho_0} \]
where $\rho_0$ is the character of $A_{N^{\vee}} = A_N$ given by
\[ \rho_0(f_i) = \epsilon(N_i^{\vee} \mu^{-2},  \psi_{-1})  = \epsilon(N_i \otimes M_1, \psi). \]
\vskip 10pt

\item $\Hom_{\U(W_{B'}}(\pi_{\rho_{N^{\vee}} \cdot \rho_0}    \otimes \omega_{\psi_{0,-2}, \mu,  W_{B'}}, \pi_{N'}^{\eta}) \ne 0$. This question was addressed in the previous section, and we deduce that the desired non-vanishing holds if and only if
the character 
\[ ( \rho_N \cdot \rho_0 , \, \,   \eta) \in \text{Irr}(A_N) \times \text{Irr}(A_{N'}) \]
is the distinguished character $\chi_0$ in [GGP, Conjecture 16.3] for the skew-hermitian case for 
$(W_{B'}, \mu)$. More precisely, the desired non-vanishing holds if and only if 
\[  \rho_N(f_i)  \cdot  \epsilon(N_i \otimes M_1, \psi)  = \epsilon(N_i^{\vee} \otimes (N')^{\vee} (\mu), \psi_{-1}) = \epsilon( N_i \otimes M_2, \psi), \]
so that
\[  \rho_N(f_i) = \epsilon(N_i \otimes M, \psi), \] 
 and
\[  \eta(e_i) = \epsilon((N'_i)^{\vee} \otimes N^{\vee}(\mu), \psi_{-1}) = \epsilon(M_{2,i} \otimes N, \psi). \]
  \end{enumerate}
\vskip 5pt

This shows that 
\[  \Hom_{\U(V_B)}(\pi_{\rho_M} \otimes \pi_{\rho_N}, \CC) \ne 0 \]
if and only if the character $\rho_M \times \rho_N$ is the distinguished character $\chi_0$ of [GGP, Conjecture 16.3], computed using the additive character $\psi$ of $k/k_0$.

 \vskip 15pt

\section{Restriction from $\U(3)$ to $\U(2)$: stable case}

We now consider the restriction problem for stable Vogan packets of $\U(3)$. We preserve the notation of the previous sections. In particular, we have the pairs of spaces $W \subset V$ and $W'  \subset V'$, with $\disc (V)=1 = \disc (V/W)$. Moreover, we use the additive character $\psi_{-2}$ to normalize the Vogan parameterization for $\U(W)$.
\vskip 5pt

Let $M$ be an irreducible 3-dimensional conjugate-orthogonal representation of $WD(k)$, so that its associated Vogan packet has the form
\[  \Pi_M = \{ \pi_M, \pi'_M \}, \] 
where $\pi_M$ is a representation of $\U(V)$ and $\pi'_M$ is the same representation considered on $\U(V')$. 
If $M$ is an irreducible representation of the Weil group $W(k)$, then the representation $\pi_M$
is supercuspidal. Otherwise, 
\[  M = \mu \boxtimes S_3 \]
where $\mu$ is a conjugate-orthogonal character of $W(k)$ and $S_3$ denotes the irreducible 3-dimensional representation of $\SL_2(\CC)$. In this case, the representation $\pi_M$ is a twisted Steinberg representation
\[  \pi_M = St \otimes (\mu \circ \det). \]
\vskip 10pt

On the other hand, let $N$ be an arbitrary  2-dimensional conjugate-symplectic representation of $WD(k)$ with associated Vogan packet $\Pi_N$ of $\U(W)$. 
We would like to determine 
\[  \Hom_{\U(W)}(\pi_M \otimes \pi_N, \CC) \]
for $\pi_M \in \Pi_M$ and $\pi_N \in \Pi_N$. 
We shall reduce this question to the case when $\Pi_M$ and $\Pi_N$ are both supercuspidal packets, by first treating the other cases directly. The supercuspidal case will then be handled by a global method.
\vskip 10pt

We first consider the case when $M = \mu \boxtimes S_3$. 
Since we can absorb the twist by $\mu$ into the parameter $N$, we may assume without loss of generality that $\mu = 1$. In this case, $\pi_M = St$ is a quotient of a (un-normalized) principal series representation:

\[  
0 \longrightarrow 1 \longrightarrow {\rm Ind}_B^{\U(V)}(1) \longrightarrow
 St \longrightarrow 0.  \]

\noindent We have:
\vskip 5pt

\begin{prop}
(i) If $N$ is not the parameter of the Steinberg representation of 
$\U(W)$, we have
\[  \Hom_{\U(W)}(St  \otimes  \pi_{\rho_N}, \CC) = 
\Hom_{\U(W)}(I_B(1) \otimes  \pi_{\rho_N} , \CC) = \Hom_{\U(L)}( \pi_{\rho_N}  , \CC). \]
 In particular, $\Hom_{\U(W)}(St \otimes  \pi_{\rho_N}, \CC) \ne 0$ if and only if
 \[  \rho_N(f_i) = \epsilon(N_i, \psi) = \epsilon(N_i \otimes M, \psi). \]
 \vskip 5pt
 
 \noindent (ii) If $N$ is the parameter of the Steinberg representation of 
$\U(W)$, so that $\Pi_N = \{ St_{\U(W)}, 1_{\U(W')} \}$, we have
 \[  \Hom_{\U(W)}(St \otimes St_{\U(W)},\CC) = \Hom_{\U(W)} (I_B(1) \otimes St_{\U(W)}, \CC)  \ne 0. \]
 On the other hand,
 \[   \Hom_{\U(W')}(1_{\U(V')} , \CC)) + \Hom_{\U(W')}(St_{\U(V')}, \CC)  \]
 is equal to
 \[   \Hom_{\U(W')} ({\rm Ind}_{B'}^{\U(V')}(1), 1_{\U(W')}) =  \Hom_{\U(L')}(1_{\U(W')}, \CC) = \CC, \]
 so that
 \[  \Hom_{\U(W')}(St_{\U(V')}, \CC) = 0. \]
 \end{prop}

   \vskip 10pt
   
 This proposition verifies [GGP, Conjecture 16.3] when $M = \mu \otimes S_3$ and $N$ is arbitrary. We may thus restrict attention to the case when $M$ is an irreducible representation of $W(k)$, so that $\Pi_M$ is a stable supercuspidal packet.
 \vskip 5pt
 
 Next, we consider the case when 
 \[  N = P + (P^{\sigma})^{\vee} \quad \text{or} \quad \mu \otimes S_2, \]
 where $P$ and  $(P^{\sigma})^{\vee}$ are not necessarily distinct.
  In such cases, the associated representations of $\U(W)$ are contained in principal series representations of $\U(W)$. Thus, we need to compute:
 \[  \Hom_{\U(W)}(\pi_M, {\rm Ind}_B^{\U(W)}(\chi)) \]
 for a supercuspidal representation $\pi_M$ of $\U(V)$. By Frobenius reciprocity, we see that this is equal to
 \[  \Hom_T ((\pi_M)_Z , \chi)  \] 
 where the unipotent radical $Z$ of $B \subset \U(W)$ coincides with the center of the unipotent radical of a Borel subgroup of $\U(V)$. But  $(\pi_M)_Z$ is isomorphic to the regular representation $\mathcal{S}(k^{\times})$ of $T \cong k^{\times}$. Thus we have 
 \[
 \Hom_{\U(W)}(\pi_M, {\rm Ind}_B^{\U(W)}(\chi)) =   \Hom_{k^{\times}} (\mathcal{S}(k^{\times}) , \chi)   = \CC. \]
 \vskip 5pt
 
 This proposition is proved by a standard application of Mackey theory. Indeed, it is a special case of [GGP, Theorem 15.1], and so we omit its proof here.
 The main point to note is that the proposition verifies [GGP, Conjecture 16.3] when
 \[   N = P + (P^{\sigma})^{\vee} \quad \text{with $P \ne (P^{\sigma})^{\vee}$,} \]
 as the principal series representation on $\U(W)$ is irreducible.
 If  $P = (P^{\sigma})^{\vee}$, the parameter $N$ is dihedral with respect to $k/k_0$ and the corresponding principal series representation of $\U(W)$ is the sum of two irreducible summands, and we have not determined which of these summands contribute to the 1-dimensional Hom space above. 
 \vskip 10pt
 
 Finally, when $N = \mu \otimes S_2$, we may assume without loss of generality
 that $\mu=1$ (by absorbing $\mu$ into $M$). Then
 \[  \Pi_N = \{ St_{\U(W)}  , 1_{\U(W')} \}, \]
and
\[ 
0 \rightarrow 1_{\U(W)} \rightarrow {\rm Ind}_B^{\U(W)} 1 \rightarrow St_{\U(W)} 
\longrightarrow 0.  \]
The above computation shows that 
\[  \Hom_{\U(W)}(\pi_M, {\rm Ind}_B^{\U(W)} ) = \CC. \]
On the other hand, by [GRS], we have 
\[  \Hom_{\U(W)}(\pi_M, \CC) = 0 \]
and
\[  \Hom_{\U(W)}(\pi'_M, \CC) = 0. \]
Indeed, if these Hom spaces were not zero, $\pi_M$ and $\pi'_M$ would be obtainable as a theta lifting from some $\U(2)$, contradicting the fact that $M$ is a stable parameter of $\U(3)$. Thus, we conclude that
\[   \Hom_{\U(W)}(\pi_M, St_{\U(W)}) \ne 0, \]
 which is what [GGP, Conjecture 16.3] predicts.
 \vskip 10pt
 
 We are thus left with the case when $M$ is an irreducible representation of $W(k)$ and $N$ is a supercuspidal parameter.    
 For this remaining case, we shall prove the following local theorem by global methods.
\vskip 10pt

\begin{thm} Let $W$ be a 2 dimensional hermitian subspace
of a hermitian space $V$ of dimension 3 over $k$. 
Suppose that $\pi_M$ (resp. $\pi_N$) is an irreducible supercuspidal
representation of $\U(V)$ (resp. $\U(W)$) with irreducible Langlands 
parameter  $M$ (resp.
$N$) of $W(k)$. Then $\epsilon(M \otimes N, \psi)$ for a character $\psi$ of 
$k/k_0$ is independent of $\psi$, so may be denoted as $\epsilon(M \otimes N)$.
Suppose that ${\rm Hom}_{\U(W)}
(\pi_M \otimes \pi_N, \CC) \not = 0.$ Then 
\begin{eqnarray*} 
\epsilon (M\otimes N) & = & \left \{ \begin{matrix} 1 
& {\sf ~if~~} \U(V) \times \U(W) {\sf ~~is ~~quasi-split}\\ 
-1 & 
{\sf ~otherwise~~}   \end{matrix}\right.
\end{eqnarray*}
\end{thm}

\noindent{\bf Remark :} The method that we follow to prove this theorem is pretty general,
but it is based on a global theorem of Ginzburg, Jiang, and Rallis [GJR3, theorem 4.6] 
which assumes that automorphic forms on unitary groups $\U(n)$ have base change to 
$\GL(n)$ something which is known at the moment only for generic automorphic representations
on quasi-split unitary groups. However, by Rogawski [Ro], base change is known for any
unitary group in 3 variables, which is why we have restricted ourselves to $\U(3)$
in the above theorem. Nonetheless, we have formulated some of the preliminary
results below in greater generality.

\vskip 5pt

We begin with the following globalization result about local fields, which will 
 be applied to globalize hermitian spaces over local fields so that  there is 
 no ramification outside the place being considered, and the
unitary groups at infinity are either compact, or of rank 1.

\begin{lemma}
Let $k$ be a quadratic extension of a non-archimedean local field $k_0$. Then 
there exists a totally real number field $F$ with $k_0$ as its completion, 
and a quadratic totally imaginary extension $E$ 
of $F$  with corresponding
completion $k$ such that $E$ is unramified over $F$ at all finite 
places different from $k$. 
\end{lemma}

\noindent{\bf Proof:} Except for the requirement about $E$ being unramified
at any place of $F$ outside the place $k$, this is well-known. Suppose that a quadratic extension $E_1$ 
over $F_1$ with possible ramifications 
is constructed. Then
a well-known technique, {\it crossing with a field}, says that after a suitable
base change, one can get rid of the ramifications; we leave the details to the reader.

\begin{lemma}Let $W$ be a hermitian space
over $k$. Let $F$ be a totally real number field with completion
$k_0$ at a place of $F$, and let $E$ be a quadratic totally imaginary extension 
of $F$  with corresponding
completion $k$. Then there is a hermitian space $\WW$ over $E$ giving rise to
$W$ over $k$ in such a way that the corresponding unitary group
is quasi-split at all finite places of $F$ except the one corresponding
to the completion $k_0$; and at all but 
one infinite place the group is the 
compact group $\U(n)$, and at the remaining infinite place,
 the group is either $\U(n)$, or  $\U(n-1,1)$; if $n$ is odd, we can assume
that the group is compact at all the infinite places. 
\end{lemma}

\noindent{\bf Proof :} The proof of the lemma will depend on the well-known
classification of a hermitian form over a number field, according to which a 
hermitian form over a number field is determined by 
\begin{enumerate}
\item the normalized discriminant, and
\item the signatures at the infinite places.
\end{enumerate}
Moreover, given any normalized discriminant, and signatures at infinite places (except for obvious 
compatibility between normalized discriminant and signatures), there is a  hermitian form. 

We also note the following exact sequence from classfield theory,
$$0 \rightarrow F^\times / \N E^\times \rightarrow \A_F^\times/\N \A^\times_E 
\rightarrow {\rm Gal}(E/F) \rightarrow 0,$$
from which it follows that one can construct an element in $F^\times$ which is trivial
in $F^\times_v/\N E^\times_v$ at all the finite places except $k_0$, and which at the infinite
places has the desired signs, except that the product of the signs is 1 or -1, depending
on whether the element in $k_0^\times/\N k^\times$ is trivial or nontrivial.

The proof of the lemma is now completed by observing that a hermitian form 
of normalized discriminant 1 over a 
non-archimedean local field 
defines a quasi-split group, and that the 
normalized 
discriminants of the
hermitian spaces $Z_1\bar{Z}_1 + Z_2 \bar{Z}_2 + \cdots + Z_n\bar{Z}_n$, and
$Z_1\bar{Z}_1 + Z_2 \bar{Z}_2 +\cdots + Z_{n-1}\bar{Z}_{n-1} -Z_n\bar{Z}_n$ over $\CC$ 
are negative  of each other, and if $n$ is odd, the normalized 
discriminants of the
hermitian spaces $Z_1\bar{Z}_1 + Z_2 \bar{Z}_2+  \cdots + Z_n\bar{Z}_n$, and
$-(Z_1\bar{Z}_1 + Z_2 \bar{Z}_2+ \cdots  +Z_n\bar{Z}_n)$  
are negative  of each other. 

\begin{cor} Let $W$  be a  hermitian 
subspace of codimension 1 of a hermitian space $V$ 
over $k$. Let $F$ be a totally real number field with completion
$k_0$ at a place of $F$, and let $E$ be a quadratic 
totally imaginary extension 
of $F$  with corresponding
completion $k$. Then there is a hermitian subspace $\WW$ of codimension 1 
of a hermitian space $\VV$  over $E$ giving rise to
$W$ and $V$ over $k$ in such a way that the corresponding unitary groups
are quasi-split at all the finite places of $F$ except the one corresponding
to the completion $k_0$;  assuming $F \not = \Q$, 
the group $\U(\VV)$ is the 
compact group $\U(n+1)$ at all but two infinite places, and at the remaining 
infinite places,
 the group is either $\U(n+1)$, or  $\U(n,1)$; the subgroup $\U(\WW)$ is compact
at all but possibly one infinite place.

\end{cor}

\noindent{\bf Proof :} Let $V = W \oplus L_c$ where $L_c$ is $k$ with the
hermitian structure $cZ\bar{Z}$, $c \in k_0^\times$. Globalize $W$ by the 
previous lemma, and globalize $c$ so that the normalized discriminant of
$V$ is 1 at all the finite places of $F$ other than $k_0$  
(so that $\U(\VV)$ is quasi-split at
all the finite places of $F$ other than $k_0$), and so that $c$ has arbitrary signs 
at infinity, with only the product of the signs pre-determined which allows
for the desired conclusion.

\vskip 5pt

We omit a proof of the following corollary of the lemma which follows exactly as in the
previous corollary. 

\begin{cor} If $n \equiv 1, 2 \mod 4$, then hermitian spaces 
$W \subset V$ of dimensions $n,n+1$ 
can be globalized keeping them positive definite at infinity, and maximally split
at all finite places other than $k$. 
For $n \equiv 2 \mod 4$, if $W$ has an isotropic subspace of dimension $n/2$, 
and for $n \equiv 1 \mod 4$, if $V$ has an isotropic subspace of dimension $(n+1)/2$, 
then there are an even number of real places in $F$, else an odd number of real
places. 
\end{cor}   

\vskip 5pt 

\noindent{\bf Proof of Theorem 14.2:} By the corollary above, we can assume
that $\U(\VV)$ is compact at infinity.
It is then easy to
see that we can  globalize the representation
$\pi_M$ of $\U(V)$ to an automorphic representation $\Pi_1$ of $\U(\VV)(\A)$
in such a way that it is unramified at all the finite places of $F$
except $k_0$.

By Lemma 1 of [P6], we can globalize $\pi_N$ to an automorphic representation
$\Pi_0$ such that the period integral 

$$\int_{\U(\WW)\backslash \U(\WW)(\A)} f_0 f_1 \not = 0,$$
for some $f_0$ in $\Pi_0$,  and $f_1$ in $\Pi_1$.

By the theorems due to Ginzburg, Jiang, and Rallis, cf. [GJR3, theorem 4.6], since the 
period integral is nonzero, the central critical $L$-value,

$$L(\frac{1}{2}, \Pi^E_0 \otimes \Pi^E_1) \not = 0,$$
where $\Pi^E_0$ and $\Pi^E_1$ denote base change of $\Pi_0$ and $\Pi_1$
to $E$.

This implies that the global root number,
$$\epsilon(\frac{1}{2}, \Pi^E_0 \otimes \Pi^E_1) = 1.$$

Let $$\Pi_0 = \otimes_w \Pi_{0,w}, {\rm~~~~~~~~~and~~~~~~~~}\Pi_1 = \otimes_w \Pi_{1,w},$$
with $\Pi_{0,v}= \pi_N,$ and $\Pi_{1,v}= \pi_M$. From the nonvanishing of the period
integral, it follows that 
$${\rm Hom}_{\U(\WW_w)}(\Pi_{0,w} \otimes \Pi_{1,w}, \CC) \not = 0$$
for all places $w$ of $F$.

Since the representations $\Pi_{1,w}$ for $w$, a finite place of $F$, not $v$, are  unramified 
by construction, they are in particular quotients of principal
series representations, and hence we know the validity of theorem 14.2 
for such representations:

$$\epsilon_w(\frac{1}{2}, \Pi^E_{0,w} \otimes \Pi^E_{1,w}) = 1$$
for all finite places $w$ of $F$ except $v$.

Since the global epsilon factor is a product of local epsilon factors,
we have
$$ \epsilon(\frac{1}{2}, M \otimes N) 
\epsilon_\infty(\frac{1}{2}, \Pi^E_{0,\infty} \otimes \Pi^E_{1,\infty}) = 1.$$

The following lemmas then complete the proof of the theorem on noting that there
are an even number of places at infinity if $\U(\WW)$ is quasi-split, and odd number
of places at infinity when $\U(\WW)$ is not quasi-split.

\begin{lemma} Let $W$ be a codimension 1 hermitian subspace
of a positive definite hermitian space $V$ of dimension $n+1$ over $\CC$. 
Suppose that $\pi_0$ (resp. $\pi_1$) is a finite dimensional irreducible
representation of $\U(V)$ (resp. $\U(W)$). Let the Langlands 
parameter of $\pi_0$ (resp. $\pi_1$)  be $\sigma_0$ (resp.
$\sigma_1$).
Suppose that ${\rm Hom}_{\U(W)}
(\pi_0 \otimes \pi_1, \CC) \not = 0.$ Then 
\begin{eqnarray*} 
\epsilon (\sigma_0 \otimes \sigma_1) & = & \left \{ \begin{matrix} 1 
& {\sf ~if~~} n \equiv 0,3 \mod 4 \\ 
-1 & 
{\sf ~if ~~} n \equiv 1,2 \mod 4.   \end{matrix}\right.
\end{eqnarray*}
\end{lemma}

\noindent{\bf Proof:} The proof of this lemma is a simple consequence of the
well-known branching law from the compact group $\U(n+1)$ to $\U(n)$, combined
with the value of the epsilon factor given by the following lemma, which has been demonstrated in Proposition \ref{P:11.1}.
\vskip 5pt

\begin{lemma}
Let $\psi$ be the additive character on $\CC$ given by
$\psi(z) =  e^{ -2\pi i y}$ where $z = x+ iy$. For $n$ an integer, 
let $\chi_n$ denote the character
 $\chi_n(z) = (\bar{z}/z)^{n/2} = e^{-ni \theta}$ for $z = re^{i\theta} \in \CC^\times$. Then for $n$ odd,

\begin{eqnarray*} 
\epsilon (\chi_n, \psi) & = & \left \{ \begin{matrix} 1 
& {\sf ~if~~} n > 0 \\ 
-1 & 
{\sf ~if~~}  n < 0. \end{matrix}\right.
\end{eqnarray*}
\end{lemma}

\begin{lemma} Let $\pi_0$ (resp. $\pi_1$) be a finite dimensional
irreducible representation of the compact group $\U(n)$ (resp. $\U(n+1)$)
with $L$-parameter restricted to $\CC^\times$ given by an $n$-tuple of half-integers
$\sigma_0 = \{- \lambda_n < -\lambda_{n-1} < \cdots < -\lambda_1 \}$  
(resp. $\sigma_1 = \{\mu_1 < \mu_2 < \cdots < \mu_{n+1}\}$ an $(n+1)$-tuple of half-integers), 
where all the $\lambda_i's$ are
half-integers but not integers if $n$ is even, and are integers if $n$ is odd, and 
$\mu_i's$ are all integers if $n$ is even, and half-integers but not integers
if $n$ is odd. Then $\Hom_{\U(n)}(\pi_1 \otimes \pi_0,\CC) \ne 0$ if and
only if $$\mu_1 < \lambda_1 < \mu_2 < \lambda_2 < \cdots < \lambda_n < \mu_{n+1}.$$  
\end{lemma}

\begin{cor} With notation as in the lemma, and assuming that $\pi_0^{\vee}$ appears in $\pi_1$  
$$\epsilon(\chi_{\mu_k} \otimes \sigma_0) = (-1)^{n-k+1}, {\sf ~~for~~all~~} k,$$
and therefore,

\begin{eqnarray*}
\epsilon(\sigma_1 \otimes \sigma_0) & = & \prod_{k=1}^{n+1}(-1)^{n-k+1} \\
                                    & = & (-1)^{\frac{n(n+1)}{2}}.
\end{eqnarray*}

\end{cor}

\vskip 20pt
\noindent{\bf Remark :} It should be mentioned that the global method 
followed in the proof of theorem 14.2 proves that if there is an invariant linear form, then the epsilon factor has some fixed value. The natural
variant of the theorem of Waldspurger in [Wa4] to unitary groups 
proves that such an 
invariant form  exists on a relevant pair of unitary groups, which will
then strenghthen theorem 14.2 to an  
if and only if statement.

\end{document}